\documentclass[11pt,letterpaper]{amsart}
\usepackage[latin1]{inputenc}
\usepackage[active]{srcltx}
\usepackage{fullpage}
\usepackage{amsfonts,enumerate}
\usepackage[mathscr]{euscript}
\usepackage{epsfig}
\usepackage{latexsym}
\usepackage{leftindex}
\usepackage[all]{xy}
\usepackage{amssymb}
\usepackage{amsthm}
\usepackage{float}
\usepackage{amsmath}
\usepackage{amsxtra}
\usepackage{xcolor}
\usepackage[colorlinks]{hyperref}
\usepackage{fancyvrb}
\usepackage{tikz}
\hypersetup{citecolor=blue}

\usepackage[pagewise]{lineno}

.tfm
.tfm

\theoremstyle{plain}
\newtheorem{proposition}{Proposition}[section]
\newtheorem{theorem}[proposition]{Theorem}
\newtheorem{corollary}[proposition]{Corollary}

\newtheorem{lemma}[proposition]{Lemma}
\newtheorem{definition}[proposition]{Definition}

\newtheorem*{conjecture}{Conjecture}

\newtheorem*{definition*}{Definition}
\newtheorem{ass}{Assumption}
\newtheorem*{theorem*}{Theorem}
\newtheorem*{lemma*}{Lemma}

\newtheorem*{prop*}{Proposition}

\theoremstyle{definition}

\theoremstyle{remark}

\newtheorem{remark}{Remark}

\numberwithin{table}{section}

\DeclareMathOperator{\Frob}{Frob}

\DeclareMathOperator{\Gal}{Gal}

\DeclareMathOperator{\Jac}{Jac}
\DeclareMathOperator{\Disc}{\Delta}
\DeclareMathOperator{\JacMot}{\mathbf J}

\newcommand{\D}{\mathcal D}

\newcommand{\F}{\mathbb F}
\newcommand{\bfC}{\mathcal C}

\newcommand{\Om}{{\mathscr{O}}}

\newcommand{\GL}{{\rm GL}}

\newcommand{\PSL}{{\rm PSL}}

\newcommand{\rhog}{\rho^{\text{geom}}}
\newcommand{\hgm}{\HGM((a,b),(c,d)|t)}
\newcommand{\hgms}{\HGM((a,b),(c,d)|t_0)}
\def\mot{\mathfrak h}
\newcommand{\fg}{\PP^1\setminus\{0,1,\infty\}}

\def\gen{\varpi}

\def\ZZ{\mathbb Z}

\def\FF{\mathbb F}
\def\QQ{\mathbb Q}

\def\PP{\mathbb P}

\def\PP{\mathbb P}
\def\AA{\mathbb A}
\def\CC{\mathbb C}

\def\pp{\mathfrak p}
\def\nn{\mathfrak n}

\def\C{\mathcal C}
\def\M{\mathscr M}

\def\<#1>{{\left\langle{#1}\right\rangle}}

\def\Z{{\mathbb Z}}            
                
\def\Q{{\mathbb Q}}             
               
\newcommand{\mubb}{\mbox{$\raisebox{-0.59ex}
  {$l$}\hspace{-0.18em}\mu\hspace{-0.88em}\raisebox{-0.98ex}{\scalebox{2}
  {$\color{white}.$}}\hspace{-0.416em}\raisebox{+0.88ex}
  {$\color{white}.$}\hspace{0.46em}$}{}}

\def\id#1{{\mathfrak{#1}}}    
\def\normid#1{{\norm{\id{#1}}}}
\DeclareMathOperator{\norm}{{\mathscr N}}

\DeclareMathOperator{\trace}{{\mathrm{Tr}}}
\newcommand{\HGM}{\mathcal H}

\usepackage{array}  
\newcolumntype{L}{>{$}l<{$}} 
\newcolumntype{C}{>{$}c<{$}}

\let\kro\dkro
\newcommand\TODO[1]{\textcolor{red}{#1}}

\newcommand{\lmfdbec}[3]{\href{http://www.lmfdb.org/EllipticCurve/Q/#1/#2/#3}{{\text{\rm#1.#2#3}}}}

\begin{document}
	
\title[HGM and Fermat]{Hypergeometric motives and the generalized Fermat equation}
	
\author{Franco Golfieri Madriaga}
\address{Center for Research and Development in Mathematics and Applications (CIDMA),
		Department of Mathematics, University of Aveiro, 3810-193 Aveiro, Portugal}
\email{francogolfieri@ua.pt}
\author{Ariel Pacetti}

\address{Center for Research and Development in Mathematics and
  Applications (CIDMA), Department of Mathematics, University of
  Aveiro, 3810-193 Aveiro, Portugal} \email{apacetti@ua.pt}
\thanks{The second author was funded by the Portuguese Foundation for
  Science and Technology (FCT) Individual Call to Scientific Employment Stimulus
  (https://doi.org/10.54499/2020.02775.CEECIND/CP1589/CT0032). This
  work was supported by CIDMA and is funded by the FCT, under grant
  UIDB/04106/2020 (https://doi.org/10.54499/UIDB/04106/2020)}

\address[Villegas]{The Abdus Salam International Centre for Theoretical Physics, Strada Costiera 11, 34151 Trieste, Italy}
\email{villegas@ictp.it}

\dedicatory{with an Appendix by Ariel Pacetti and Fernando Rodriguez Villegas}


\begin{abstract}
  In the beautiful article \cite{Darmon} Darmon proposed a program to
  study integral solutions of the generalized Fermat equation
  $Ax^q+By^p+Cz^r=0$. In the aforementioned article, Darmon proved many
  steps of the program, by exhibiting models of
  hyperelliptic/superelliptic curves lifting what he called ``Frey
  representations'', Galois representations over a finite field of
  characteristic $p$. The goal of the present article is to show how
  hypergeometric motives are more natural objects to obtain the global
  representations constructed by Darmon, allowing to prove most steps
  of his program.
\end{abstract}
	
\maketitle

\section{Introduction}

The study of Diophantine's equations is probably one of the oldest
research areas in mathematics, which got a lot of attention after Fermat's statement that the equation
\[
x^n + y^n = z^n
\]
does not have any non-trivial solution if $n \ge 3$ (non-trivial
meaning none coordinate equal to $0$). After Wiles' groundbreaking
proof of Fermat's statement (in \cite{Wiles}) a lot of attention has
been given to the study of the so called \emph{generalized Fermat
  equation}.  Let $A,B,C$ be non-zero pairwise coprime integers and
let $p,q,r$ be positive integers. The \emph{generalized Fermat
  equation} is the affine surface given by the equation
\begin{equation}
  \label{eq:genFermat}
 Ax^q + By^p + Cz^r=0.
\end{equation}
A challenging problem is that of understanding its set of integral
points.  The name ``generalized Fermat equation'' comes from the
observation that setting $A=B=C=1$ and $p=q=r$ we recover Fermat's
classical equation. However, Fermat's equation determines a curve in
the projective plane, while the generalized Fermat equation determines
a surface in the affine plane, making its study a harder problem.  Two
well known conjectures regarding solutions to the generalized Fermat
equation are the following.

\begin{conjecture}[Beal's conjecture]
For any triple of exponents $(p,q,r)$, with
$p,q,r \ge 3$, the equation
\[
x^q + y^p = z^r,
\]
has no primitive non-trivial solution.
\end{conjecture}
Recall the following definition.
\begin{definition*}
  A solution $(\alpha,\beta,\gamma)$ of (\ref{eq:genFermat}) is called
  primitive if $\gcd(\alpha,\beta,\gamma)=1$.
\end{definition*}
Beal's conjecture (with a one million dollars reward) is deeply related
to the so called \emph{Fermat-Catalan} conjecture.
\begin{conjecture}[Fermat-Catalan's conjecture]
The set of triples $(\alpha,\beta,\gamma)$ which are primitive solutions of the equation
\[
x^q + y^p = z^r,
\]
for any triple of exponents $(q,p,r)$ with
$\frac{1}{q}+\frac{1}{p}+\frac{1}{r}<1$ is finite.
\end{conjecture}
Actually there exists a candidate for the finite list of solutions (as
given in page 515 of \cite{DarmonGran}), based on numerical
experiments.  The restriction $p,q,r \ge 3$ in Beal's conjecture has
to do with the geometry of the surface $S$ defined by
equation~(\ref{eq:genFermat}). The \emph{characteristic} $\chi$
of the surface $S$ is defined by
\begin{equation}
\chi := \frac{1}{q} + \frac{1}{p} + \frac{1}{r}-1.
\label{eq:char}
\end{equation}
The case $\chi >0$ is similar to the case of a genus zero curve,
namely it has either zero or infinitely many integral points (as
proved by Beukers in \cite{Beukers}). In practice, there is a finite
list of exponents $(p,q,r)$ with $\chi >0$, allowing to study each of them individually:
\begin{itemize}
\item The solutions to
(\ref{eq:genFermat}) for
$(q,p,r)= (2,2,r\geq 2), (3,3,2), (2,3,4) , (2,4,3)$ were given in
\cite{Beukers} using Zagier's factorization.
\item  The remaining case where
$\chi >1$ corresponds to $(q,p,r)=(2,3,5)$ studied by Thiboutot
(and completely solved in \cite{MR2070150}).
\end{itemize}

Similarly, one can study the finite list of exponents $(q,p,r)$ having
characteristic $0$. They correspond (up to a change of variables in
equation (\ref{eq:genFermat})) to the triples
$(2,3,6),(2,6,3),(3,3,3),$ $(4,4,2)$ and $(4,2,4)$. The set of
solutions for each of them is well known:
\begin{itemize}
\item The case $(3,3,3)$ corresponds to the Fermat cubic (hence there
  are no non-trivial solutions).
  
\item The case $(2,6,3)$ has no non-trivial solution, while the case
  $(3,2,6)$ has the unique non-trivial solution $(-2,\pm 3,1)$ as
  proved by Bachet (see \S 6.2 of \cite{DarmonGran}).
  
\item The case $(2,4,4)$ has no non-trivial proper solution (as proved
  by Leibniz) while the case $(4,4,2)$ was proved by Fermat.
\end{itemize}
We suggest readers interested in the subject to look at
\cite{DarmonGran} (\S 6) for solutions to the general Fermat equation
when $\chi = 0$ (i.e. when one allows general coefficients in the
equation).

Understanding solutions to the generalized Fermat equation when
$\chi <0 $ is a very challenging open problem. A beautiful and
deep result of Darmon and Granville (\cite{DarmonGran}) asserts that
equation (\ref{eq:genFermat}) has finitely many \emph{primitive}
integral points for each choice of the parameters
$A,B,C,p,q,r$. However, the proof is non-effective (as it depends on
Faltings' proof of Mordell's conjecture).

\vspace{3pt}

\noindent{\bf Support for Fermat-Catalan's conjecture: the ABC conjecture}.

\vspace{1pt}

The well known ABC conjecture (due to Oesterl\'e and Masser) implies a
strong finiteness result: there exist only finitely many triples
$(A,B,C)=(x^q,y^p,z^r)$ satisfying $A+B=C$ and
$\gcd(A,B,C)=1$. Therefore the values of $(x,y,z,p,q,r)$ for which a
non-trivial solution exists belong to a finite set (see
\cite{DarmonGran}). Unfortunately (as is widely believed within the
mathematical community) the ABC conjecture is far from being
proved, hence a different approach is needed.

\vspace{3pt}

\noindent{\bf The modular method}.

\vspace{1pt} A very powerful method to study Fermat's equation is the
so called \emph{modular method} (as developed by Frey, Hellegouarch,
Mazur, Ribet, Wiles, Taylor et al.) The modular method allowed to
prove non-existence of solutions for different exponents, which can be
divided into two cases (see Tables 1, 2 and 3 of \cite{MR3280939} for
references):
\begin{itemize}
\item Specific exponents, like $(2,3,n)$ for $n=7, 8, 9, 10, 11, 15$,
  $(3,4,5)$, $(3,5,5)$,$(5,5,7)$ $(5,5,19)$ and $(7,7,5)$.
  
\item Families of exponents, like $(2,4,n)$, $(2,6,n)$, $(2,2n,3)$,
  $(2,2n,9)$, $(3,6,n)$, $(n,n,2)$, $(n,n,3)$, $(2n,2n,5)$,
  $(2l,2m,n)$ (for $l,m$ primes and $n=3, 5, 7, 11$), $(2l,2m,13)$.
\end{itemize}
The modular method ends relating a solution of a Diophantine equation
with an object (an automorphic form) belonging to a finite dimensional
vector space, which (in principle) can be computed. It is not our goal
to describe the modular method in great detail, but we content
ourselves with a very simplified description of the steps involved:
\begin{enumerate}[(I)]
\item Attach to a putative solution $P=(\alpha,\beta,\gamma)$ of
  (\ref{eq:genFermat}) a geometric object $\bfC$ defined over a small
  degree number field $K$. The object should have the property that
  the reduction of its Galois representation modulo a well chosen
  prime $\id{p}$ has a small ramification set (independent of the
  solution $P$).  In most well known cases (like Fermat's last
  theorem) the geometric object $\bfC$ is an elliptic curve, either
  defined over $\Q$ or over a number field. In the remarkable article
  \cite{Darmon}, Darmon attaches to a solution a
  hyperelliptic/superelliptic curve of $\GL_2$-type\footnote{To our
    knowledge there is no application of the modular method where the
    geometric objects is not related to a $2$-dimensional Galois
    representations.} for different families of exponents (see
  also \cite{2205.15861}). When possible, the field $K$ should be
  totally real, since in this case many different modularity theorems
  are known, and Hilbert modular forms are easier to compute.
  
\item Study the compatible system of Galois representations
  $\{\rho_{\bfC,\lambda}:\Gal(\overline{\Q}/K)\to
  \GL_2(\overline{\Z_{\lambda}}))\}$ attached to $\bfC$ (coming from
  the action of the Galois group on its \'etale cohomology), for
  $\lambda$ ranging over prime ideals of the ``coefficient
  field''. The final goal of this second step is to prove that the
  family is \emph{modular}, i.e. it matches the Galois representation
  attached to an automorphic form of weight $k$.

\item Prove that the reduction of the $\id{p}$-th member of the family
  (for a proper choice of the prime ideal $\id{p}$ as in $(1)$) is absolutely
  irreducible, and compute its conductor $\id{n}$. Then a variant of
  Ribet's \emph{lowering the level result} (proved in \cite{MR1047143}
  for classical modular forms, and a stronger version in
  \cite{MR3294620} for Hilbert modular forms) implies the existence of
  an automorphic form $F$ of level $\id{n}$ and weight $k$ whose
  residual Galois representation $\overline{\rho_{F,\id{p}}}$ is
  isomorphic to $\overline{\rho_{\bfC,\id{p}}}$.
  
\item Compute the space of automorphic forms of level $\id{n}$ and weight
  $k$ and prove that none of them can be related to a solution of
  (\ref{eq:genFermat}).
\end{enumerate}

Despite our naive and simplified description, each step of the
program has its subtleties! Even when the first three steps of the
program can be successfully applied (and it seems like a new
Diophantine equation will be solved), it usually happens that
the last step fails for one of the following two reasons: either the space
of automorphic forms has huge dimension (hence we encounter a
computational problem), or the existence of ``trivial'' (or small)
solutions to our equation makes the elimination step (IV) to fail.

It is sometimes possible to overcome the second failure by applying
the so called \emph{multi-frey method} as developed by Siksek in
\cite{MR2414954}. Basically, the idea is to construct many different
objects in step (I) attached to the same solution, and make use of
them all in the elimination step (IV). Usually the construction of the
geometric object $\bfC$ in step (I) is an artisan's work, making the
modular method (and the multi-frey method) hard to apply in new
instances.

In the beautiful article~\cite{Darmon}, Darmon introduced the notion
of \emph{Frey representations} that gives a general framework to
construct the Galois representations of step (I) as we now recall. If
$F$ is a field of characteristic zero, we denote by $\Gal_F$ the Galois
group $\Gal(\overline{F}/F)$.

Let $K$ be a
number field, let $K(t)$ be the function field in one variable over
$K$ and let $\FF$ be a finite field of characteristic $p$.

\begin{definition*}
  A Frey representation is a continuous Galois representation
\[
\rho_t : \Gal_{K(t)} \to \GL_2(\FF),
\]
satisfying the following conditions:
\begin{enumerate}
\item The restriction of $\rho$ to $\Gal_{\overline{K}(t)}$ has trivial determinant and is irreducible.
  
\item Let $\overline{\rho}^{\text{geom}}$ denote the projectivization
  of the restriction of $\rho$ to $\Gal_{\overline{K}(t)}$. Then
  $\overline{\rho}^{\text{geom}}$ is unramified outside
  $\{0,1,\infty\}$ and it maps the inertia groups at $0, 1$ and
  $\infty$ to subgroups of $\PSL_2(\FF)$ of order $p$, $q$ and $r$
  respectively.
\end{enumerate}
\end{definition*}

By a result of Beckmann (\cite{MR1116916}, see also \cite[Lemma
1.2]{Darmon}) Frey's representation $\rho_t$ specialized at the point
$t_0 =-\frac{A\alpha^q}{C\gamma^r}$ (where $(\alpha,\beta,\gamma)$ is a
putative solution to (\ref{eq:genFermat})) is unramified outside the
set of primes dividing $ABCpqr$ (a number which only depends on the
equation, not on the solution).

The representation $\rho_{t_0}$ should match the reduction of the
representation $\rho_{\bfC}$ coming from the geometric object
constructed in the first step of the modular method.

In \cite{Darmon} Darmon studies for each exponents $(q,p,r)$ all
irreducible Frey representations $\rho_t$ with inertia orders $q,p,r$
at $0,1,\infty$ respectively (the representation depends
only on the exponents, not on a particular solution to the
equation).
%
However, Frey's representations are not enough to prove non-existence
of solutions, mainly because of the following two problems:
\begin{enumerate}
\item How can we compute all number fields fixed by the kernel of $\rho_{t_0}$?
  
\item How to discard fields that are not related to solutions?
\end{enumerate}

The standard way to solve the first problem is to relate the
representation to a modular one (going back to step (II) of the
modular method). This would be automatically true if Serre's
conjectures are true for general number fields (or at least totally
real ones). Since nowadays Serre's conjectures are only proven for the
rational field, we need to construct ``lifts'' of Frey's
representation $\rho_{t_0}$ to fields of characteristic zero, and
prove their modularity.

In \cite{Darmon}, Darmon study some particular families of exponents
and constructs for each of them a curve $\bfC$ of $\GL_2$-type
satisfying that the reduction of its Galois representation $\rho_\bfC$
matches $\rho_{t_0}$. Furthermore (under some suitable hypothesis)
Darmon is able to prove modularity of the global representation $\rho_\bfC$,
fulfilling steps (I) and (II) of the modular method for them
(see \cite{2205.15861} for results on step (III) for exponents
$(q,q,p)$).

Hypergeometric motives (HGM for short) are useful while studying the
general case of the generalized Fermat equation; they are a source of
motives whose Galois representations are ``reasonably`` well
understood.  They are mentioned in Darmon's original article and in
\cite{MR1709312} Darmon proved that all representations coming from
hypergeometric motives defined over totally real fields are
modular. The catch is that on doing so Darmon assumed the veracity of
some very strong modularity lifting results which (unfortunately)
are still far from being proven.

There is a difference between our approach and Darmon's one (that
allows us to prove stronger results). What Darmon calls an
\emph{hypergeometric abelian variety} (in \cite{MR1709312}) are
varieties whose existence depend on results of Belyi (Theorems 1 and 2
of \cite{MR534593}). Belyi's construction is not explicit, so it does
not give any information on traces of Frobenius endomorphisms nor a
description of the variety at primes of bad reduction of the
variety. Then even when one could theoretically use Belyi construction
to complete step (III) of the modular method, it will be of no use for
the last one.

During the last years the theory of hypergeometric motives became a
very active research area. The rational ones have nice models (as
explained in \cite{MR2394437}) and many of their expected properties
are known (see \cite{MR4442789} for a nice introduction). The purpose of the
present article is to explain how the use of rank two
hypergeometric motives (as studied in \cite{GPV}) can be used to
fulfill steps (I), (II) and (IV) of the modular method, allowing to study
solutions of new families of the generalized Fermat
equation~(\ref{eq:genFermat}).

%
%
%
Instead of presenting an hypergeometric motives as triples of
monodromy matrices (as done by Darmon), we present them as pairs of
rational numbers, $(a,b),(c,d)$ following the more standard
convention. With this description it is easy to verify that there are
finitely many rational rank $2$ hypergeometric motives, corresponding
to rational elliptic curves. In Table~\ref{table:rational} we list all
rational rank 2 hypergeometric motives satisfying that at least one
monodromy matrix has infinite order. Doing reverse engineering, we can
deduce for each parameter which family of exponents
of~(\ref{eq:genFermat}) can be studied with it. This recovers most of the
elliptic curves appearing in the literature while studying the
generalized Fermat equation (see Table~\ref{table:ell-curve}).

The modular method is well suited to study solutions of
(\ref{eq:genFermat}) when the exponents lie in a line $L$ (Fermat's
last theorem corresponding to the line $q=p=r$). If we restrict to prime
exponents, then there are four different types of lines (up to a
relabeling of the variables):
\begin{enumerate}
\item $L_1:\{x=y=z\}$, corresponding to Fermat's last theorem of
  exponents $(p,p,p)$.
  
\item $L_2:\{x=y, z=r\}$, corresponding to exponents $(p,p,r)$.
  
\item $L_3:\{x=z=q\}$, corresponding to exponents $(q,p,q)$.
  
\item $L_4:\{x=q,z=r\}$, corresponding to exponents $(q,p,r)$.
\end{enumerate}
The recipe for constructing an hypergeometric motive (or its
parameters) defined over a totally real number field $K$ (independent
of the varying parameter $p$) consists on finding four rational numbers
$a,b,c,d$ satisfying the following properties:
\begin{itemize}
\item The numbers $a-b$, $a-d$, $c-b$ and $c-d$ are not integers (for
  the monodromy representation to be irreducible).
\item The numbers $a+b$ and $c+d$ are integers (for the motive to be
  defined over a totally real number field).
  
\item If the parameter $x$ (respectively $z$) varies, then $c=d$
  (respectively $a=b$) and they belong to
  $\{\frac{1}{2},1\}$. Otherwise, the denominator of $c$ belongs to
  $\{x,2x\}$ (respectively $\{z, 2z\}$).
\end{itemize}

\begin{definition}
\label{defi:motive-GFE}  
For each one of the previous lines define the following motives:
\begin{enumerate}
\item The motive
  $\HGM_1^-:=\HGM((\frac{1}{2},\frac{1}{2}),(1,1)|t)$. This motive is
  isomorphic to the elliptic curve with equation
$E:y^2=x(x-1)(1-tx).$
\item The motives defined over $\Q(\zeta_r)^+$:
  \begin{itemize}
  \item $\HGM_2^+:=\HGM((\frac{1}{r},-\frac{1}{r}),(1,1)|t)$,
  \item $\HGM_2^-:= \HGM((\frac{1}{2r},-\frac{1}{2r}),(1,1)|t)$.
  \end{itemize}
  
\item The two families of motives defined over
  $\Q(\zeta_q)^+$:
  \begin{itemize}
  \item $\HGM_{3,s}^+:=
    \HGM((\frac{1}{q},-\frac{1}{q}),(\frac{s}{q},-\frac{s}{q})|t)$,
    for $s \in \FF_q^\times$, $s \neq \pm 1$,
    
  \item
    $\HGM_{3,s}^-:=\HGM((\frac{1}{2q},-\frac{1}{2q}),(\frac{s}{q},-\frac{s}{q})|t)$,
    for $s \in \FF_q^\times$.
  \end{itemize}

\item The two families defined over the composition of $\Q(\zeta_q)^+$
  and $\Q(\zeta_r)^+$:
  \begin{itemize}
  \item $\HGM_{4}^+:=\HGM((\frac{1}{r},-\frac{1}{r}),(\frac{1}{q},-\frac{1}{q})|t)$,
    
  \item
    $\HGM_{4}^-:=\HGM((\frac{1}{2r},-\frac{1}{2r}),(\frac{1}{q},-\frac{1}{q})|t)$.
  \end{itemize}
\end{enumerate}
\end{definition}
The notation is chosen so it is consistent with \cite{Darmon}; the $+$
motives are tame at primes dividing $2$, while the $-$ motives are (a
priory) wild.  We will prove that these motives are suitable for the
modular method. The proof's strategy can be summarized in the
following diagram: to a solution $(\alpha,\beta,\gamma)$ to
(\ref{eq:genFermat}) one attaches the hypergeometric motive $\hgms$
specialized at $t_0=-\frac{A\alpha^p}{C\gamma^r}$. Then one studies
two different types of congruences
\begin{figure}[H]
  \xymatrix{
& & & &     \HGM^1 & & \HGM_1 \ar@{~>}[rd] & & \\
& & & &     \HGM^2 & \hgms \ar@{~>}[l] \ar@{~>}[lu] \ar@{~>}[ru] & & \HGM_2 \ar@{~>}[ru]
  }
  
  \caption{\label{fig:congruences} Congruences}
\end{figure}
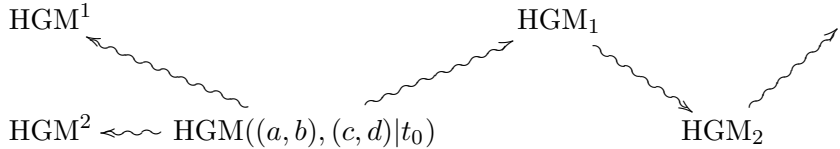
The left congruences are used to prove properties of the residual
representation (namely step (1) of the modular method), while the
right ones are used to prove modularity (namely step (2)) via the
standard ``propagation of modularity'' approach.

There is a catch here: to propagate modularity (using some modularity
lifting theorem \`a la Wiles) and to lower the level we need ``large
residual image'' (as in step (3)). Although large image results are
expected to hold for geometric representations, there are few
instances where unconditional results are proven (and mostly for
representations coming from elliptic curves). Even when some
modularity results for reducible residual representations are known
(as in \cite{MR1793414} and \cite{MR4467307}), an absolutely
irreducible image result is needed in all lowering the level
results. This is probably the largest missing step needed to make
Darmon's program to work in general. In some particular instances (for
example when the HGM is related via a congruence to an elliptic curve)
we can circumvent this issue, as will be later explained.

We must warn the reader that some properties of hypergeometric motives
are still not completely understood. If we specialize the parameter
$t$ at a particular value $t_0$, ramification at the so called
``tame'' primes (primes dividing the numerator or denominator of
$t_0(t_0-1)$ but not dividing $N$) is well understood. However, the
behavior at the so called ``wild'' primes (primes dividing $N$) is
not.

In order to make the modular method work, we also need a bound at the
wild primes. For the motives $\mot_{i}^{\pm}$, such bounds can be
obtained using the recent article \cite{2503.21568} (build upon
results of \cite{2410.21134}). In this way, we completely fulfill the
three stated steps of the modular method.

To keep the article concise, we focus on setting the bases to
accomplish the stated steps of the modular method rather than proving
non-existence of solutions for new families of exponents.  In the
article~\cite{Lucas} we will apply the present ideas to prove non-existence
of solutions to the generalized Fermat equation for exponents
$(3,5,p)$.

%
\vspace{2pt}

The article is organized as follows: Section~\ref{section:HGM}
consists of a user's guide to hypergeometric motives. It contains a
brief presentation of the theory (including their definition)
together with their main properties needed in the present article. Their
proofs are mostly given in \cite{GPV}.

Section~\ref{section:rational-HGM} contains the study of rational
hypergeometric motives. In particular, we show how the complete list
of rational HGM (listed in Table~\ref{table:rational}) appear in the
literature while studying families of the generalized Fermat
equation.  Here is a summary of the results obtained in this section.

\begin{theorem*}
  Table~\ref{table:ell-curve} is a complete table of all rational rank $2$
  hypergeometric motives with a monodromy matrix of infinite order (up
  to quadratic twists). The table includes the hypergeometric parameter, an
  equation of the corresponding elliptic curve, the exponents of the
  generalized Fermat equation where it can be used together with its appearance
  in the literature.

\begin{table}
\begin{tabular}{|c|l|c|c|}
  \hline
Parameters & Curve & Exponents &Literature\\
\hline
$(1/2,1/2),(1,1)$ &$y^2=x(x-1)(1-tx)$ & $(p,p,p)$&\cite{Wiles}\\
$(1/3,2/3),(1,1)$ &$y^2+xy+\frac{t}{27}y=x^3$  &  $(p,p,3)$  & \cite{MR1468926}\\
$(1/4,3/4),(1,1)$ &$y^2+xy=x^3+\frac{t}{64}x$ & $(p,p,2)$ & \cite{MR1468926}\\
$(1/6,5/6),(1,1)$ &$y^2+xy=x^3-\frac{t}{432}$ & $(p,p,3)$ & \cite{Darmon}\\
$(1/3,2/3),(1/4,3/4)$&$y^2=x^3-12tx-16t^2$ & $(2,p,3)$ &\cite{DarmonGran}\\
$(1/6,5/6),(1/3,2/3)$ &$y^2=x^3-3t^3x+t^4(t+1)$ & $(3,p,3)$ &\cite{MR1618290} \\
  \hline
\end{tabular}
\caption{Table of rational rank $2$ hypergeometric motives}
\label{table:ell-curve}
\end{table}
\end{theorem*}

Section~\ref{section:applications} contains a detailed analysis of the
ramification of the motives $\mot_i^{\pm}$ given in
Definition~\ref{defi:motive-GFE} (see
Theorem~\ref{thm:global-conductor}). It also contains the value of the
residual conductor of their attached residual Galois representations
(Theorem~\ref{thm:residual-conductor}) and the ``finite at $p$''
condition (Theorem~\ref{thm:finite}). In particular, we prove that our
hypergeometric motives satisfy all the required properties of step (1)
in the modular method.

Section~\ref{section:modularity} is devoted to prove modularity of our
motives. Here is a summary of the results we can prove:
\begin{itemize}
\item the motive $\mot_2^+$ is modular when $r \mid \alpha \gamma$
  (Theorem~\ref{thm:mod-2+}),
  
\item the motive $\mot_2^-$ is modular if $r\ge5$ (Theorem~\ref{thm:mod-2-}),
  
\item the motives $\mot_{3,s}^+$ are modular if $q \mid \gamma$ (Theorem~\ref{thm:mod-3+}),

\item the motives $\mot_{3,s}^-$ are modular if $q\ge5$ (Theorem~\ref{thm:mod-3-}),
\item the motives $\mot_{4}^+$ are modular if $q\mid \alpha$ and
  $r \mid \gamma$ (Theorem~\ref{thm:mod-4-+}),
  
\item the motives $\mot_{4}^-$ are modular if either $r\ge5$ and
  $q \mid \alpha$ or $q\ge5$ and $r \mid \gamma$
  (Theorem~\ref{theorem:modularity}).

\end{itemize}
We do not expect any of the required hypothesis to be needed. There
are two particular families where unconditional results can be proven
(as done in Section~\ref{section:particular-cases}): the family
$(2,p,r)$ and the family $(3,r,p)$. More concretely
\begin{itemize}
  
\item The motive $\mot_{4}^-$ for the family $(2,p,r)$ is modular
  if $r \nmid A$ and $r \ge 11$ (see
  Theorem~\ref{thm:modularity-2pr-}).
  
\item The motive $\mot_{4}^+$ for the family $(3,p,r)$ is modular
  if $r \nmid A$ and $r \ge 5$ (see
  Theorem~\ref{thm:modularity-3pr+}).
  
\item The motive $\mot_{4}^-$ for the family $(3,p,r)$ is modular
  if $r \nmid A$ and $r \ge 11$ (see
  Theorem~\ref{thm:modularity-3pr-}).
\end{itemize}

Section~\ref{section:wild} studies the behavior of the motives
$\mot_i^{\pm}$ at wild primes. For an odd wild prime $\id{q}$, we
prove that the conductor exponent is at most $3$ when
$\id{q} \nmid ABC$ (see Corollary~\ref{coro:odd-bound}). At primes
$\id{q}$ dividing $2$, we can prove that the conductor is bounded by
$6$ when $\id{q} \nmid AC$, assuming that $p,q,r \ge 5$ (see
Corollary~\ref{coro:bound-2}). The explicit bounds obtained allow to
fully apply the modular method to new families of generalized Fermat
equations (assuming the veracity of step (3)). We are left to
compute newforms of finitely many spaces of Hilbert modular forms,
and prove that they are not related to solutions.

The last section (section~\ref{section:elimination}) explains how to
use hypergeometric motives to perform an ``elimination''
procedure due to Mazur (to our knowledge the unique systematic
strategy to perform step (4) in the modular method approach). It is an
interesting problem to decide whether other techniques (like the
symplectic method) can be adapted to hypergeometric motives.

We include an appendix (due to the second named author and Fernando
Rodriguez Villegas) which proves the following result.

\begin{theorem*}
  Let $N$ be a positive integer different from $1$. Then the
  hypergeometric motive with parameter $(\frac{1}{N},-\frac{1}{N}),(1,1)$ is part of
  the middle cohomology of an explicit hyperelliptic curve.
\end{theorem*}
As a corollary (which motivated the study and proof of the result) we
deduce that the Galois representation attached to the motive
$\mot_2^{\pm}$ coincides with Darmon's ones studied in \cite{Darmon}
(see Corollaries~\ref{coro:comparison-Darmon} and
\ref{coro:comparison-Darmon-2}). Still, we believe that the result
might be of independent interest.

\vspace{3pt}

\noindent {\bf Acknowledgments:} We want to express our gratitude to
Fernando Rodriguez Villegas for many helpful conversations as well as
his detailed explanation of various properties satisfied by
hypergeometric motives. Special thanks go to David Roberts as well,
for many fruitful conversations. At last, but not least, we want to
thank Nuno Freitas and Lucas Villagra Torcomian for many suggestions
on an earlier version of the article.

\section{User's guide to Hypergeometric motives}
\label{section:HGM}
For a nice introduction to the hypergeometric motives used in this
article see \cite{MR4442789},\cite{BH},\cite{katz},\cite{MR4493579}
and \cite{GPV}. If $a$ is a rational number, by $\exp(a)$ we denote the
root of unity $e^{2 \pi i a}$.

\subsection{Hypergeometric motives and monodromy
  representations}
\label{sec:monodromy}
The input of a (rank two) hypergeometric motive is a
\emph{parameter} consisting of two pairs $(a,b)$, $(c,d)$ of rational
numbers (up to translation by integers).

\begin{definition}
  The parameter $(a,b),(c,d)$ is called \emph{generic} if no
  element of the set $\{a-c,a-d,b-c,b-d\}$ is an integer.
\end{definition}
From now on, we assume that all parameters are generic (the
non-generic case is more subtle and will not be needed).  The generic
condition allows to consider the parameters $a,b,c,d$ as elements in
$\Q/\Z$. Set $\gamma = c+d-a-b$. For a parameter $(a,b),(c,d)$, define
the following matrices:
\begin{equation}
  \label{eq:M0}
      M_0 := \begin{cases}
    \begin{pmatrix}\exp(-c) & 0 \\ 0 &\exp(-d)\end{pmatrix} & \text{ if }c-d \not \in \Z,\\
    \exp(-c)\begin{pmatrix}1 & 1 \\ 0 &1\end{pmatrix} & \text{ if }c-d  \in \Z,      
    \end{cases}
\end{equation}
\begin{equation}
  \label{eq:M1}
      M_1 := \begin{cases}
    \begin{pmatrix}1 & 0 \\ 0 &\exp(\gamma)\end{pmatrix} & \text{ if }\gamma \not \in \Z,\\
    \begin{pmatrix}1 & 1 \\ 0 &1\end{pmatrix} & \text{ if }\gamma  \in \Z,
    \end{cases}
\end{equation}
and
\begin{equation}
  \label{eq:Minfty}
      M_\infty = \begin{cases}
    \begin{pmatrix}\exp(a) & 0 \\ 0 &\exp(b)\end{pmatrix} & \text{ if }a-b \not \in \Z,\\
    \exp(a)\begin{pmatrix}1 & 1 \\ 0 &1\end{pmatrix} & \text{ if }a-b  \in \Z.      
    \end{cases}
\end{equation}
Let $N$ be the least common denominator of $a,b, c, d$. 
\begin{theorem}
  \label{thm:monodromy-representation}
  Let $(a,b),(c,d)$ be generic parameters, and let $x \in \fg$. Then
  there is a monodromy representation
  \[
    \rho_t:\pi_1(\fg,x) \to \GL_2(\Z[\zeta_N]),
  \]
  whose monodromy matrices at $0$, $1$ and $\infty$ are conjugate to
  $M_0$, $M_1$ and $M_\infty$ respectively.
\end{theorem}
\begin{proof}
  The result is due to Levelt, see for example \cite[Theorem 3.5]{BH}.
\end{proof}
For $s \in \PP^1$, denote by $I_s$ an inertia group of $\Gal_{\overline{\Q}(t)}$ at $s$.
For each prime ideal $\id{p}$ of $\Q(\zeta_N)$, the monodromy
representation extends (by continuity) to a geometric representation
\begin{equation}
  \label{eq:geom-rep}
  \rhog_t : \Gal_{\overline{\Q}(t)} \to \GL_2(\Z[\zeta_N]_{\id{p}}),  
\end{equation}
satisfying
\begin{enumerate}
\item The representation $\rhog_t$ is unramified outside the points $\{0,1,\infty\}$.
  
\item For $s \in \{0,1,\infty\}$,
  \[\rhog_t(I_s) \simeq \langle M_s\rangle.
  \]
\end{enumerate}

\begin{remark}
  More generally, one can study representations of the projective line
  $\PP^1$ (with coefficients in $\CC$ or in a finite field) which are
  unramified outside finitely many points and having fixed (up to
  conjugation) monodromy matrices. The case of three ramified points
  is simpler, since this case is ``rigid'' (i.e. the representation is
  determined by the conjugacy class of the matrices $M_0$, $M_1$ and
  $M_\infty$ as explained in Remark 2.10 of \cite{GPV}).
  \label{remark:HGM-cover}
\end{remark}

The geometric representation is expected to match the (restriction to
$\Gal_{\overline{\Q}(t)}$ of a) representation of a pure motive
defined over a cyclotomic field (see \cite{MR534593}). In the rank two
case the motive can be explicitly constructed: it is (up to a twist by
a Hecke character) part of the new middle cohomology of the so called
``Euler's curve''. Explicitly, let $(a,b),(c,d)$ be a generic
parameter, and let $N$ be their least common denominator. Define the
quantities:
\begin{equation}
  \label{eq:ABCD-parameters}
A=(d-b)N, \qquad B=(b-c)N, \qquad C=(a-d)N, \qquad D=dN,
\end{equation}
and define Euler's curve by
\begin{equation}
  \label{eq:curve}
\bfC: y^N = x^A(1-x)^B(1-tx)^Ct^D.  
\end{equation}
The curve $\bfC$ is geometrically irreducible precisely when
$\gcd(A,B,C,N)=1$. In this case, the new part of the
$\zeta_N$-eigenspace of the first \'etale cohomology group of $\bfC$
is a motive $\M_\bfC$ defined over $F=\Q(\zeta_N)$. The hypergeometric
motive $\hgm$ is a twist of $\M_\bfC$ by a Jacobi motive that we now
describe.

\subsubsection{The Jacobi motive}
Let $\id{q}$ be a prime ideal of $F=\Q(\zeta_N)$ not dividing $N$, and
let $\FF_q$ denote its residue field (so $N \mid
(q-1)$). Following~\cite{MR0051263}, for $x$ an integer prime to
$\id{q}$, let $\chi_{\id{q}}(x)$ denote the $N$-th root of unity
congruent to $x^{(q-1)/N}$ modulo $\id{q}$. Extend the definition by
setting $\chi_{\id{q}}(x)=0$ if $\id{q} \mid x$. This determines a
character of order $N$
\begin{equation}
  \label{eq:char-def}
  \chi_{\id{q}}: (\ZZ[\zeta_N]/\id{q})^{\times} \rightarrow
  \CC^{\times}.
\end{equation}
Let $\psi$ be an additive character of $\FF_q$
\begin{definition}
  \label{def:jacobi-motive}
  Let ${\bf a}=(a_1,\ldots,a_r)$ and ${\bf b}=(b_1,\ldots,b_s)$ be two
  sets of rational numbers and let $N$ be their least common
  denominator. The Jacobi motive attached to ${\bf a},{\bf b}$ at a
  prime ideal $\id{q}$ of $K=\QQ(\zeta_N)$ not dividing $N$ is the sum
  \begin{equation}
    \label{eq:Jac-motive}
    \JacMot({\bf a, b})(\id{q}) = (-1)^{r+s+1} \frac{g(\psi,\chi_{\id{q}}^{Na_1})\cdots
    g(\psi,\chi_{\id{q}}^{Na_r})g(\psi,\chi_{\id{q}}^{N(\sum_j b_j-\sum_i a_i)})}{g(\psi,\chi_{\id{q}}^{Nb_1})\cdots g(\psi,\chi_{\id{q}}^{Nb_s})}.
  \end{equation}
\end{definition}
The value $\JacMot({\bf a, b})(\id{q})$ does not depend on the
choice of the additive character $\psi$.  In \cite{MR0051263} Weil
proves that a Jacobi motive is a Hecke character, i.e. there exists a
character $\chi: \AA_{\Q(\zeta_N)} \to \CC^\times$ such that for all
but finitely many prime ideals $\id{q}$ of $K$,
\[
  \JacMot({\bf a,b})(\id{q})=\chi(\id{q}).
\]
\begin{definition}
  Suppose that Euler's curve $\bfC$ is irreducible. Then the hypergeometric motive
  $\hgm$ is defined by
\begin{equation}
  \label{eq:trace-Frob}
  \hgm := \M_\bfC \otimes \JacMot((-a,-b,c,d),(c-b,d-a))^{-1} (-1)^{d-b},    
\end{equation}
where $(-1)^{d-b}$ denotes the quadratic character of $\AA_F$, whose value at a
prime ideal $\id{q}$ not dividing $N$ equals
$\gen(-1)^{(d-b)(\normid(q)-1)}$, for $\gen$ any generator of the
character group of $\Om_F/\id{q}$.
\end{definition}
\begin{remark}
  All hypergeometric motives considered in the present article satisfy
  that Euler's curve is irreducible, so the previous definition
  applies.  When Euler's curve $\bfC$ is reducible, the hypergeometric
  motive is defined (see Definition 6.27 of \cite{GPV}) as a twist of
  an irreducible Euler's curve (with parameters $(a-d,b-d),(c-d,1)$)
  by another Jacobi motive.
\end{remark}
It follows from its definition that $\hgm$ is defined over the field
$F$.
\begin{lemma}
  \label{lemma:quad-char}
  Let $(a,b),(c,d)$ be a generic parameter.
  \begin{enumerate}
  \item If $v_2(d-b)=0$, the character $(-1)^{d-b}$ is trivial.
    
  \item If $v_2(d-b)=-1$, the character $(-1)^{d-b}$ matches the
  quadratic character $\kro{-1}{\id{q}}$.
  \end{enumerate}
  \begin{proof}
    Let $\id{q}$ denote a prime of $F$ not dividing $2N$, and let
    $\FF_{\id{q}}$ denote the finite field $\Om_F/\id{q}$. Then
    $\gen(-1)$ is $1$ when $-1$ is a square in $\FF_{\id{q}}$, or
    equivalently
    \[
      \gen(-1) =
      \begin{cases}
        1 & \text{ if } \normid{q} \equiv 1 \pmod 4,\\
        -1 & \text{ if } \normid{q} \equiv 3 \pmod 4.
      \end{cases}
      \]
      When $d-b$ has odd denominator the character $(-1)^{d-b}$ is
      clearly trivial (since $\normid{q}-1$ is even).  When $d-b$ has
      even denominator (of valuation $1$), the exponent
      $(d-b)(\normid{q}-1)$ is odd precisely when
      $\normid{q} \equiv 3 \pmod 4$, hence the result.
  \end{proof}
\end{lemma}
The motives $\mot_i^+$, for $i=1,2,3,4$ satisfy the first hypothesis
of the lemma, while the motives $\mot_i^-$ satisfy the second one.
\begin{remark}
  In the present article we will consider parameters satisfying the
  extra condition $a+b$ and $c+d$ being integers. In this case the
  Jacobi motive $\JacMot(-a,-b,c,d),(c-b,d-a)$ is just a Tate twist
  (so can be mostly be ignored).
\end{remark}

\begin{theorem}
  The Galois representation attached to $\hgm$ restricted to $\Gal_{\overline{\Q}(t)}$ is isomorphic to $\rhog_t$.
\end{theorem}

\begin{proof}
  See \cite[Theorem 7.13]{GPV}.
\end{proof}

Permuting the order of the parameters in the first or the second
entries of the parameter give isomorphic motives, i.e. the motives
$\hgm$, $\HGM((b,a),(c,d)|t)$, $\HGM((a,b),(d,c)|t)$ and
$\HGM((b,a),(d,c)|t)$ are all isomorphic. We will make constant use of
this fact while listing hypergeometric motives.

\subsection{Field of definition}
\label{sec:field-def}
Keeping the previous notation, let $(a,b),(c,d)$ be a generic
parameter and let $N$ denote their least common denominator. Let $H$ be the group
\begin{equation}
  \label{eq:H-defi}
  H:=\{ r \in (\Z/N)^\times \; : \; \{a,b\} = \{ar, br\} \text{ and }\{c,d\} = \{cr, dr\}\},
\end{equation}
i.e. $H$ is the set of elements that leave the sets $\{a,b\}$ and $\{c,d\}$
stable under multiplication (recall that our parameters are elements
in $\Q/\Z$). By \cite[Lemma 8.1]{GPV}, the group $H$ is a subgroup of
$\Z/2 \times \Z/2$.

The group $H$ can be canonically identified with a subgroup of
$\Gal(\Q(\zeta_N)/\Q)$. Let $K$ be the field $\Q(\zeta_N)^H$ fixed
by the subgroup $H$. The following result is proven in \cite[Theorem
8.2]{GPV}.

\begin{theorem}
  \label{thm:extension}
  If $a+b$ and $c+d$ are integers and $\bfC$ is irreducible, then the
  motive $\hgm$ is defined over $K$. In particular, for every prime
  ideal $\id{p}$ of $K$, the geometric Galois representation $\rhog_t$
  extends to a representation
\begin{equation}
  \label{eq:extension}
  \rho_{t,\id{p}}:\Gal_{K(t)}\to \GL_2(K_{\id{p}}).
\end{equation}
\end{theorem}

Specializations of $\rho_{t,\id{p}}$ provide compatible families of
Galois representations as will be explained in the next section.
Although the hypothesis $a+b$ and $c+d$ being integers seems quite
restrictive, most motives defined over totally real fields satisfy
this condition. 

\subsection{Specializations}
\label{sec:specialization}
For $(a,b),(c,d)$ generic parameters and $t_0$ a rational number,
$t_0 \neq 0,1$, we can consider the specialization of the motive
$\hgm$ at $t_0$ (the values $0$ and $1$ correspond to singular values
of the motive). The combinatorial nature of an hypergeometric motive
allows to obtain a lot of information of the specialized motive,
namely:
\begin{itemize}
\item \emph{A description of the set of bad places:} it consists of prime
  ideals dividing the numerator/denominator of $t_0$ or $t_0-1$ but not
  dividing $N$ (the so called \emph{tame} primes) together with prime ideals
  dividing $N$ (the so called \emph{wild} primes).
  
\item \emph{A description of inertia at tame primes}.
  
\item \emph{The trace of Frobenius at good primes}.
\end{itemize}

Fix an additive character $\psi$ on $\FF_q$. For $\omega \in \widehat{\FF_q^\times}$, a character of $\FF_q^\times$, denote by $g(\psi,\omega)$ the Gauss sum
\begin{equation}
\label{gsum}
g(\psi,\omega) = \sum_{x \in \FF_q^\times} \omega(x)\psi(x).
\end{equation}
%
Let $\id{q}$ be a prime ideal of $F=\Q(\zeta_N)$ and let $\FF_q$ be
its residual field. Let $\chi_{\id{q}}$ be the character of $\FF_q$
defined in (\ref{eq:char-def}).

\begin{definition}[Finite hypergeometric sum]
For $t_0 \in \FF_q$, define the finite hypergeometric series $H_{\id{q}}((a,b),(c,d)|t_0)$ by
\[
H_{\id{q}}((a,b),(c,d)|t_0)=\frac{1}{1-q} \sum_{\omega \in \widehat{\FF_q^\times}} 
\frac{g(\psi,\chi_{\id{q}}^{-aN}\omega) g(\psi,\chi_{\id{q}}^{cN} \omega^{-1})}{g(\psi,\chi_{\id{q}}^{-aN})
	g(\psi,\chi_{\id{q}}^{cN})}\frac{g(\psi,\chi_{\id{q}}^{-bN}\omega) g(\psi,\chi_{\id{q}}^{dN}\omega^{-1})}{g(\psi,\chi_{\id{q}}^{-bN})
	g(\psi,\chi_{\id{q}}^{dN})}\omega(t_0).
  \]
  \label{defi:finite-hgs}
\end{definition}
The value $H_{\id{q}}((a,b),(c,d)|t_0)$ does not depend on the choice
of the additive character $\psi$.  If $\id{p}$ is a prime ideal of
$F$, denote by $I_{\id{p}}$ an inertia group of $\Gal_F$ at
$\id{p}$.

\begin{theorem}
  \label{thm:hgm-specialization}
  Let $(a,b),(c,d)$ be a generic parameter satisfying that $a+b$ and
  $c+d$ are integers.  Let $t_0 \in \Q$ be a rational number and let
  $K=\Q(\zeta_N)^H$. Then $\hgms$ provides a compatible system of
  irreducible Galois representations
  \[
\{\rho_{t_0,\id{p}}:\Gal_{\Q(\zeta_N)} \to \GL_2(\Q(\zeta_N)_{\id{p}})\}_{\id{p}},
\]
where $\id{p}$ ranges over
  primes ideals of $\Q(\zeta_N)$. The family satisfies the following properties:
  \begin{enumerate}
  \item Let $\id{q}$ be a prime ideal of $\Q(\zeta_N)$ not dividing
    $N$ nor the norm of $\id{p}$. Let $\Frob_{\id{q}}$ be a
    Frobenius element. Then if $v_{\id{q}}(t_0(t_0-1))=0$, the
    representation $\rho_{t_0,\id{p}}$ is unramified at $\id{q}$ and
    \[
      \trace{\rho_{t_0,\id{p}}(\Frob_{\id{q}})} = H_{\id{q}}((a,b),(c,d)|t_0).      
\]
    
  \item If $\id{q}\nmid \id{p}N$ and $r=v_{\id{q}}(t_0) > 0$, then
    $\rho_{t_0,\id{p}}(I_{\id{q}}) \simeq \langle M_0^r \rangle$ (up to conjugation).
  \item If If $\id{q}\nmid \id{p}N$ and $r=v_{\id{q}}(t_0) < 0$, then $\rho_{t_0,\id{p}}(I_{\id{q}}) \simeq \langle M_\infty^r \rangle$ (up to conjugation).
  \item If If $\id{q}\nmid \id{p}N$ and $r=v_{\id{q}}(t_0-1) > 0$, then $\rho_{t_0,\id{p}}(I_{\id{q}}) \simeq \langle M_1^r \rangle$ (up to conjugation).
  \item The family extends to a family $\{\widetilde{\rho}_{t_0,\id{p}}:\Gal_K \to \GL_2(K_{\id{p}})\}$, where
    $\id{p}$ ranges over prime ideals of $K$.
  \end{enumerate}
\end{theorem}
\begin{proof}
  The first statement is proved in \cite[Theorems 4.8 and 7.23]{GPV};
  the three statements regarding inertia groups are proved in
  \cite[Appendix A]{GPV} when $M_i^r$ is the identity and in general
  in \cite{LGP}. The last statement is proved in \cite[Corollary
  8.19]{GPV}.
\end{proof}
\begin{remark}
  A script written by Fernando Ridriguez Villegas to compute
  $p$-adically $H_{\id{q}}((a,b),(c,d)|t_0)$ in terms of the $p$-adic
  Gamma function is available at the github repository
\begin{center}
\url{https://github.com/frvillegas/frvmath}.  
\end{center}

\end{remark}

As mentioned in the introduction, the result does not provide any
information at wild primes.  There is a caveat in the statement: the
motive is defined over $K=F^H$, but the formula relating the trace of
a Frobenius element to an hypergeometric motive is only proved over
$F$ (although it is expected to hold over $K$ as well).

\begin{remark}
  The last result implies that even when $v_{\id{q}}(t_0)$ is
  positive, the motive $\hgms$ might have good reduction at $\id{q}$
  (precisely when the matrix $M_0$ has finite order dividing
  $v_{\id{q}}(t_0)$). In this case, one can still give a formula to
  compute the trace of a Frobenius endomorphism at $\id{q}$ (assuming
  $\id{q}\nmid N$), as described in \cite[Appendix A]{GPV}.
\end{remark}

\begin{remark}
  The normalization of the motive $\hgms$ given in \cite{GPV} is
  useful to work with any rank $2$ hypergeometric motive, but while
  working over totally real fields it might be a Tate twist of the
  expected Hilbert modular form (when the motive happens to be
  modular). It is the case that when $a+b$ and $c+d$ are integers:
  \begin{itemize}
  \item if $c,d \in \ZZ$ (or $a,b \in \ZZ$), then no twist is needed,
    
  \item otherwise the motive $\hgms(1)$ has
    Hodge-Tate weights $\{0,1\}$.
  \end{itemize}
  We will disregard this subtlety, but the reader should keep it in
  mind while matching the motive to an automorphic form.
\end{remark}

There is an isomorphism of motives (whose veracity follows from a
simple manipulation of the finite hypergeometric sum definition)
\begin{equation}
  \label{eq:change-param}
  \hgms \simeq {\HGM((c,d),(a,b)|t_0^{-1})}.
\end{equation}
This justifies the ambiguity while choosing the order of the pairs of parameters.

\subsection{Congruences}
Let $p$ be a rational prime. For $a,b$ rational numbers, denote by
$a\sim_p b$ the relation defined by the condition that the denominator
of $a-b$ is a $p$-th power, extended naturally to vectors
component-wise.

\begin{theorem}
  \label{thm:congruences}
  Let $p$ be a prime number, and let $(a,b),(c,d)$ and
  $(a',b'),(c',d')$ be generic parameters. Let $F$ be the composite of
  the cyclomotic fields where both motives are
  defined. If $(a,b) \sim_p(a',b')$ and $(c,d) \sim_p (c',d')$ then
  both motives over $F$ are congruent modulo $\id{p}$, for $\id{p}$
  any prime ideal of their coefficient field dividing $p$.
  \end{theorem}

\begin{proof}
  See \cite[Theorem 10.3]{GPV}.
\end{proof}

\subsection{Galois action}
Keeping the previous notation, let $(a,b),(c,d)$ be a generic rational
parameter with common denominator $N$. Let $t_0 \in \QQ$ and
$\sigma \in \Gal_{\Q}$ with $\sigma(\zeta_N)=\zeta_N^j$ for some $j$
coprime with $N$. Then
\begin{equation}
  \label{eq:galois-action}
\hgms^\sigma=\HGM((ja,jb),(jc,jd)|t_0).  
\end{equation}
The result follows from \cite[Proposition 6.3]{GPV}. This allows to
chose the parameters so that (up to Galois conjugation) the numerator
of $a$ equals $1$.

\subsection{Quadratic Twists}
\label{section:quad-twist}
Let $(a,b),(c,d)$ be generic parameters such that $a+b$ and $c+d$ are integers.
\begin{proposition}
  Under the previous hypothesis, the motive
  $\HGM((a+1/2,b+1/2),(c+1/2,d+1/2)|t_0)$ is (up to a Tate twist) a
  quadratic twist by $\sqrt{t_0}$ of the motive $\hgms$.
\label{prop:quad-twist}
\end{proposition}

\begin{proof}
  By Theorem~\ref{thm:hgm-specialization} it is enough to prove the
  stated relation for the finite hypergeometric sums
  $H_q((a,b),(c,d)|t_0)$ and
  $H_q((a+1/2,b+1/2),(c+1/2,d+1/2)|t_0)$. It follows from
  \cite[Proposition 6.3]{GPV} that
  \[
    H_q((a+1/2,b+1/2),(c+1/2,d+1/2)|t_0)=\omega(t_0)^{\frac{q-1}{2}} H_q((a,b),(c,d)|t_0) \kappa,
  \]
  where $\omega$ is a generator of $\FF_q^\times $ (so $\omega^{\frac{q-1}{2}}$ equals the quadratic character of $\FF_q^\times$) and
  \[
\kappa=\frac{g(\psi,\omega^{(a+1/2)(q-1)})g(\psi,\omega^{(b+1/2)(q-1)})g(\psi,\omega^{-(c+1/2)(q-1)})g(\psi,\omega^{-(d+1/2)(q-1)})}{g(\psi,\omega^{a(q-1)})g(\psi,\omega^{b(q-1)})g(\psi,\omega^{-c(q-1)})g(\psi,\omega^{-d(q-1)})}.
\]
Since $a+b \in \ZZ$, $\omega^{a(q-1)} = (\omega^{b(q-1)})^{-1}$. Recall that
\[
  g(\psi,\omega)g(\psi,\omega^{-1})=
  \begin{cases}
    \omega(-1)q & \text{ if }\omega \text{ is non-trivial},\\
    1 & \text{ if }\omega \text{ is trivial}.\\
  \end{cases}
\]
Then if neither $a,b,c,d$ are integers nor half integers,
$\kappa=1$. Otherwise, if two are integers or half integers,
$\kappa=q^{\pm 1}$.
\end{proof}

\section{Rational rank $2$ hypergeometric motives}
\label{section:rational-HGM}
The combinatorial nature of hypergeometric motives makes it clear that
there are only finitely many rank $2$ rational ones: let $N$ denote
the least common multiple of the denominators of the parameter
$(a,b),(c,d)$. If the motive is rational then $H$ (a subgroup of
$\Z/2 \times \Z/2$) must equal the group $\Gal(\Q(\zeta_N)/\Q)$, so
$N\in \{1,2,3,4,6,8, 12\}$ (but it is not hard to verify that $N$
cannot be $8$). A small computer search gives a complete list of rational
hypergeometric motives with at least one monodromy matrix of infinite
order as given in Table~\ref{table:rational}. The table includes the
order of inertia at the three ramified points and the exponents of the
generalized Fermat equation it allows to study (each case will be
studied in detail).
\begin{table}
\begin{tabular}{|C|C|C|C|C|}
  \hline
\text{Parameters} & \text{Order at }0 & \text{Order at }1 &\text{Order at }\infty & \text{Equation}\\
\hline
  (1/2,1/2),(1,1) & \infty & \infty & \infty & (p,p,p)\\
  (1/3,2/3),(1,1) & \infty & \infty & 3 & (p,p,3)\\
  (1/4,3/4),(1,1) & \infty & \infty & 4 & (p,p,2)\\
  (1/6,5/6),(1,1) & \infty & \infty & 6 & (p,p,3)\\
  (1/3,2/3),(1/4,3/4) & 3 & \infty & 4 & (2,p,3)\\
  (1/6,5/6),(1/3,2/3) & 3 & \infty & 6 & (3,p,3)\\
\hline
  (1/6,5/6),(1/2,1/2) & \infty & \infty & 6 & (p,p,3)\\
  (1/4,3/4),(1/2,1/2) & \infty & \infty & 4 & (p,p,2)\\  
  (1/3,2/3),(1/2,1/2) & \infty & \infty & 3 & (p,p,3)\\
  (1/6,5/6),(1/4,3/4) & 4 & \infty & 6 & (2,p,3)\\
  \hline
\end{tabular}
\caption{Rational HGM with a not-finite order monodromy matrix}
\label{table:rational}
\end{table}

Note that the seventh HGM (the one after the horizontal line) is a
quadratic twist of the second one (obtained by adding $1/2$ to all
parameters), the eighth is a quadratic twist of the third one and so
on, hence only the first six HGM give different elliptic curves (up to
isomorphism).  The explicit equation we use for the elliptic curve
attached to the listed HGM were found by Cohen in \cite{Cohen} (see
also \cite{Martin}).

\subsection{The motive $\HGM ((1/2,1/2),(1,1)|t)$:}
\label{section:legendre}
corresponds to the family of elliptic curves in Legendre equation
\begin{equation}
  \label{eq:legendre}
  E_t:y^2=x(x-1)(1-tx),
\end{equation}
or equivalently the equation
\[
  E:y^2=x(x+1)(x+t).
\]
The curve has full $2$-torsion. Its invariants are
\[
\Disc(E_t)=2^4t^2(t-1)^2, \qquad j(E_t)=\frac{2^8(t^2-t+1)^3}{t^2(t-1)^2}.
\]
All odd primes dividing the numerator of $t(t-1)$ are primes of
multiplicative reduction, while the odd primes dividing the
denominator of $t$ have inertial type a ramified quadratic twist of a
multiplicative reduction type (i.e. are quadratic twist
of a Steinberg representation). The reduction type at $2$ is a little
more complicated to describe.

\begin{lemma}
  Let $t_0 \in \QQ$ be a rational number with $t_0 \neq 0,1$. The
  conductor exponent $s$ of $E_{t_0}$ at the prime $2$ is given by:
  \begin{itemize}
  \item If $v_2(t_0) \ge 0$, then
    \[
     s =
     \begin{cases}
       5 & \text{ if }t_0 \equiv 2, 3 \pmod 4,\\
       4 & \text{ if }t_0 \equiv 1 \pmod{4}, \\
       3 & \text{ if }v_2(t_0)=2,3,\\
       0 & \text{ if }v_2(t_0)=4,\\
       1 & \text{ if }v_2(t_0)\ge 5.
     \end{cases}
   \]
   
 \item If $v_2(t_0) <0$, then
   \[
     s=
     \begin{cases}
       6 & \text{ if } 2\nmid v_2(t_0),\\
       4 & \text{ if } 2 \mid v_2(t_0) \text{ and }t_0/2^{v_2(t_0)} \equiv 3 \pmod 4,\\
       3 & \text{ if } v_2(t_0)=-2 \text{ and } 4t_0 \equiv 1 \pmod 4,\\
       0 & \text{ if } v_2(t_0)=-4 \text{ and } 16t_0 \equiv 1 \pmod 4,\\
       1 & \text{ if } 2 \mid v_2(t_0)<-4 \text{ and }t_0/2^{v_2(t_0)} \equiv 1 \pmod 4.
     \end{cases}
\]
   \end{itemize}
\label{lemma:valuation-2}
\end{lemma}
\begin{proof}
  Follows from applying Tate's algorithm \cite{MR0393039} to the
  different cases.
\end{proof}

If $(\alpha,\beta,\gamma)$ is a solution to the generalized Fermat
equation $Ax^p+By^p+Cz^p=0$, set $t_0=-\frac{A\alpha^p}{C\gamma^p}$, and
consider the twisted curve
\[
  E:C\gamma^p y^2 = x(x-1)\left(1+\frac{A\alpha^p}{C\gamma^p}x\right).
\]
The twist is needed because the monodromy matrix at $\infty$ equals
$(-1)\left(\begin{smallmatrix} 1 & 1\\ 0 & 1\end{smallmatrix}\right)$, hence we need to cancel the extra $-1$. After a change of variables we
obtain the famous Frey elliptic curve (\cite{Wiles})
\[
y^2= x(x-A\alpha^p)(x+C\gamma^p).
\]

\subsection{The motive $\HGM((1/3,2/3),(1,1)|t)$:}
\label{section:p-p-3}
corresponds to the elliptic curve with equation
\begin{equation}
  \label{eq:3-p-p-1}
E_t: y^2+xy+\frac{t}{27}y=x^3.
\end{equation}
The point $P=(0,0)$ is a $3$-torsion point of $E_t$. The curve invariants
are
\[
\Disc(E_t) = -3^{-9}t^3(t-1), \qquad j(E_t)= \frac{3^3(8t-9)^3}{t^3(t-1)}.
\]
To a putative solution $(\alpha, \beta,\gamma)$ of the equation
$Ax^p+By^p+Cz^3=0$ corresponds the specialization
$t_0=-\frac{A\alpha^p}{C\gamma^3}$, giving the elliptic curve
\[
  E:y^2- 3C\gamma xy+A\alpha^pC^2 y=x^3,
\]
with
\[
\Disc(E)=3^3A^3BC^8\alpha^{3p}\beta^p, \qquad j(E)=-\frac{3^9C\gamma^3(8A\alpha^p+9C\gamma^3)}{A^3B\alpha^{3p}\beta^p}.
\]
The curve has multiplicative reduction at primes dividing $A$ and $B$,
and its local type at primes dividing $C$ is a Principal Series given by
a character of order $3$. The ramification at the prime $3$ can be
obtained running Tate's algorithm to the equation. When $A=B=C=1$ the
curve matches the one studied by Darmon and Merel in \cite{MR1468926}.

\subsection{The motive  $\HGM((1/4,3/4),(1,1)|t)$:}
\label{subsection:1/4}
corresponds to the elliptic curve with equation
\begin{equation}
  \label{eq:p-p-2}
  E_t:y^2+xy=x^3+\frac{t}{64}x.
\end{equation}
The point $P=(0,0)$ is a $2$-torsion point of $E_t$. Its invariants are
\begin{equation}
  \label{eq:1/4-inv}
\Disc(E_t)=-2^{12}t^2(t-1),\qquad j(E_t)=\frac{2^6(3t-4)^3}{t^2(t-1)}.  
\end{equation}
Since the monodromy matrix at the point $\infty$ has order $4$
(instead of $2$), we need to take a quadratic twist of the
equation.  To a putative solution $(\alpha,\beta,\gamma)$ of equation
$Ax^p+By^p+Cz^2=0$ one attaches the elliptic curve
\[
E:y^2 = x^3 - \frac{2^23^3(3A\alpha^p+4C\gamma^2)}{C}x + \frac{2^43^3(9A\alpha^p+8C\gamma^2)\gamma}{C},
\]
with invariants
\[
\Disc(E)=-\frac{2^{12}3^{12}A^2B\alpha^{2p}\beta^p}{C^3},\qquad j(E)=\frac{2^6(3A\alpha^p+4C\gamma^2)^3}{A^2B\alpha^{2p}\beta^p}.
\]
Primes dividing $A$ and $B$ are of multiplicative reduction, while
primes dividing $C$ have inertial type that of a Principal series
given by an order $4$ character. Special care needs to be taken at
primes dividing $2$ and $3$. When $A=B=C=1$, the curve matches the one
studied by Darmon and Merel in \cite{MR1468926}.

\subsection{The motive $\HGM((1/6,5/6),(1,1)|t)$:} corresponds to the elliptic curve with equation
\begin{equation}
  \label{eq:p-p-3-2}
  E_t:y^2+xy = x^3-\frac{t}{432}.
\end{equation}
It has invariants
\[
\Disc(E_t)= \frac{-t(t-1)}{2^43^3}, \qquad j(E_t)=\frac{-2^43^3}{t(t-1)}.
\]
One again, to a putative solution $(\alpha,\beta,\gamma)$ of the
equation $Ax^p+By^p+Cz^3=0$, one attaches the elliptic curve obtained by
specialization at $t_0=-\frac{A\alpha^p}{C\gamma^3}$ twisted by
$\gamma$, i.e.
\[
E:y^2=x^3-27C^4\gamma^2x-54(2A\alpha^p+C\gamma^3)C^5,
\]
with invariants
\[
\Disc(E)=2^83^9ABC^{10}\alpha^p\beta^p,\qquad j(E)=\frac{2^43^3C^2\gamma^6}{AB\alpha^p\beta^p}.
\]
Primes dividing $A$ and $B$ (but not dividing $6$) have multiplicative
reduction, while primes dividing $C$ are primes with inertial type a
Principal series given by an order $6$ character. The reduction type
at the primes $2$ and $3$ can be determined using Tate's
algorithm. This elliptic curve matches the hyperelliptic curve $C^-$
of \cite{Darmon} for $q=3$.

\subsection{The motive $\HGM((1/3,2/3),(1/4,3/4)|t)$:} corresponds to
the elliptic curve with equation
\begin{equation}
  \label{eq:2-3-p}
  E_t:y^2=x^3-12tx+16t^2,
\end{equation}
with invariants
\[
\Disc(E_t)=-2^{12}3^3t^3(t-1), \qquad j(E_t)=\frac{-2^63^3}{(t-1)}.
\]
Historically, one studies the equation $Ax^2+By^3 = Cz^p$ (i.e. the
equation for exponents $(2,3,p)$) instead of the equation with exponents
$(2,p,3)$. If $(\alpha,\beta,\gamma)$ is a putative solution to
$Ax^2+By^3=Cz^p$, then $(\alpha,\gamma,\beta)$ is a solution to
$Ax^2+(-C)z^p +By^3=0$. Specializing (\ref{eq:2-3-p}) at
$t_0=-\frac{A\alpha^2}{B\beta^3}$ and twisting by $\alpha$ gives the
curve
\[
E: y^2=x^3+12AB^3\beta x + 16A^2B^4\alpha,
\]
with invariants
\[
\Disc(E)=-2^{12}3^3A^3B^8C\gamma^p, \qquad j(E)=\frac{1728B\beta^3}{C\gamma^p}.
\]
This curve is the quadratic twist by $\sqrt{2B}$ of the elliptic curve
associated to the equation $Ax^2+By^3=Cz^p$ by Darmon and Granville in \cite{DarmonGran}.

\subsection{The motive $\HGM((1/6,5/6),(1/3,2/3)|t)$:} corresponds to
the elliptic curve with equation
\begin{equation}
  \label{eq:3-p-6}
E_t:y^2=x^3-3t^3x+t^4(t+1),
\end{equation}
with invariants
\[
\Disc(E_t)=-2^43^3t^8(t-1)^2, \qquad j(E_t)=\frac{-2^83^3t}{(t-1)^2}.
\]
Once again, instead of studying the equation $Ax^3+By^3=Cz^p$, we
consider the equation $Ax^3+(-C)z^p+By^3=0$.
%
Let $(\alpha,\beta,\gamma)$ be a putative solution of
$Ax^3+By^3=Cz^p$.
Specializing (\ref{eq:3-p-6}) at
$t_0=-\frac{A\alpha^3}{B\beta^3}$ and twisting by $\beta$ (to lower
the order of the monodromy matrix from $6$ to $3$) we obtain the elliptic curve
\[
E: y^2=x^3+3A^3B\alpha\beta x+A^4B(B\beta^3-A\alpha^3),
\]
with invariants
\[
\Disc(E)=-432A^8B^2C^2\gamma^{2p,} \qquad j(E)=\frac{6912AB\alpha^3\beta^3}{C^2\gamma^{2p}}.
\]

When $A=B=C=1$, this curve matches the one studied by Kraus in \cite{MR1618290}
(see equation $(4-3)$).

\section{The generalized Fermat equation}
\label{section:applications}
The goal of this section is to prove that the hypergeometric motives
of Definition~\ref{defi:motive-GFE} are suitable for the modular
method. Keeping the introduction's notation, let $(q,p,r)$ be the
exponents of (\ref{eq:genFermat}) lying in one of the four lines
$L_1,\ldots,L_4$. To give unified statements, denote by
$\HGM_i^{\pm}(t)$ the motive (or family of motives) attached to the
line $L_i$ as given in Definition~\ref{defi:motive-GFE} (although they
might depend of an extra parameter $s$, their conductor will not). Let
$K=F^H$ denote its field of definition.

For $d$ an integer, let $\theta_d$ be the quadratic character attached
to the extension $\Q(\sqrt{d})/\Q$. By abuse of notation, we use the
same symbol to denote the quadratic character of $\Gal_\Q$ (from class
field theory) and its restriction to any subgroup $\Gal_L$ for $L$ a
number field.  Let $(\alpha,\beta,\gamma)$ be a primitive solution of
the equation and set $t_0:=-\frac{A\alpha^q}{C\gamma^r}$. To ease
notation, we introduce
\begin{equation}
  \label{eq:motive-notation}
  \mot_i^+:=\HGM_i^+(t_0),\qquad \mot_i^-:=\HGM_i^-(t_0)\otimes \theta_\gamma.
\end{equation}

\begin{theorem}
\label{thm:global-conductor}
  Let $\ell$ be a prime ideal of $K$. Then
  \begin{itemize}
  \item If $\ell \nmid 2ABCp\alpha\beta\gamma$ then
    $\mot_1^-$ is unramified at $\ell$.
    
  \item If $\ell \nmid ABCpr\alpha\beta$ then $\mot_2^+$ is
    unramified at $\ell$.
  \item If $\ell \nmid 2ABCpr\alpha\beta$ then $\mot_2^-$ is
    unramified at $\ell$.
    
  \item If $\ell \nmid ABCpq\beta$ then $\mot_{3}^+$
    is unramified at $\ell$.
    
  \item If $\ell \nmid 2ABCpq\beta$ then $\mot_{3}^-$
    is unramified at $\ell$.
    
  \item If $\ell \nmid ABCpqr\beta$ then $\mot_{4}^+$ is unramified at $\ell$.
  \item If $\ell \nmid 2ABCpqr\beta$ then
    $\mot_{4}^-$ is unramified at $\ell$.
  \end{itemize}
\end{theorem}

\begin{proof}
  Let $\ell$ be a prime number. If $v_{\ell}(t_0(t_0-1))=0$ and
  $\ell \nmid pqr$ (or equivalently
  $\ell \nmid ABCpqr\alpha\beta\gamma$) then $\HGM_{i}^+(t_0)$ is
  unramified at $\ell$ by Theorem~\ref{thm:hgm-specialization}. The
  same is true for $\HGM_{i}^-(t_0)$ if we furthermore assume that $\ell \nmid 2$.

  Suppose then that $\ell \mid \alpha$, but $\ell\nmid Apqr$, so
  $q \mid v_{\ell}(t_0)$. The parameter given in
  Definition~\ref{defi:motive-GFE} makes (by construction) the
  monodromy matrix $M_0$ (as defined in (\ref{eq:M0})) of the motive
  $\HGM_i^{\pm}(t_0)$ to have order $q$ if $q \neq p$ and order
  $\infty$ if $q=p$. Then the second statement of
  Theorem~\ref{thm:hgm-specialization} implies that
  $\HGM_i^{\pm}(t_0)$ is unramified at primes dividing $\alpha$ when
  $q \neq p$ (proving the last four cases).

  When $\ell \mid \gamma$ the situation  is similar, but the matrix $M_\infty$ has order:
  \begin{itemize}
  \item $p$ for the motive $H_i^+(t_0)$ when $r \neq p$,
    
  \item $2p$ for the motive $H_i^-(t_0)$ when $r \neq p$,
    
  \item $\infty$ when $r=p$.
  \end{itemize}
  The same proof as before implies that the motive $\HGM_i^+(t_0)$ is
  unramified at primes dividing $\gamma$ for $i=2,3,4$. The monodromy
  matrix $M_\infty$ of the motive $\HGM_i^-$ has order $2p$, so the
  quadratic twist $\HGM_i^-(t_0)\otimes \theta_\gamma$ has order $p$
  (as occurs in the proof of Fermat's last theorem). Then the motive
  $\HGM_i^-(t_0)\otimes \theta_\gamma$ is also unramified at primes
  dividing $\gamma$ for $i=2,3,4$.

  The matrix $M_1$ always has infinite order, so the representation is
  ramified at all prime ideals dividing $\beta$ but not dividing
  $pqr$.
\end{proof}

Let $\id{p}$ be a prime ideal of $K$ dividing $p$. Then the
$\id{p}$-th residual representation attached to the motive
$\mot_i^{\pm}$ has the expected ramification.

\begin{theorem}
  \label{thm:residual-conductor}
  Let $\rho_{\id{p}}^{\pm}:\Gal_K \to \GL_2(\overline{\Q_p})$ be the
  $\id{p}$-th Galois representation attached to the motive
  $\mot_i^{\pm}$ respectively. Let
  $\ell$ be a prime ideal of $K$ not dividing $2ABCpqr$. Then the
  residual representation $\overline{\rho_{\id{p}}^{\pm}}$ is
  unramified at $\ell$. Furthermore, if $\ell\mid 2$, the
  representation $\overline{\rho_{\id{p}}^{+}}$ is also unramified at
  $\ell$.
\end{theorem}

\begin{proof}
  The last theorem implies that we only need to prove the result for
  primes $\ell$ of $K$ dividing $\alpha\beta\gamma$ but not dividing
  $2ABCpqr$. Suppose that $\ell \mid \alpha$, and consider the motive
  $\mot_1^-$. Define the motive
  $\mot:=\HGM((\frac{1}{2}+\frac{1}{p},\frac{1}{2}),(1,1)|t_0)\otimes
  \theta_\gamma$. Theorem~\ref{thm:congruences} implies that over the
  field $\Q(\zeta_p)$, the motives $\mot$ and $\mot_1^-$ are congruent
  modulo $\id{p}$. Since $v_{\ell}(\alpha) \equiv 0 \pmod p$, the
  motive $\mot$ is unramified at $\ell$ (by
  Theorem~\ref{thm:hgm-specialization}), hence
  $\overline{\rho_{\id{p}}^-}$ is unramified at $\ell$ while
  restricted to $\Gal_{K(\zeta_p)}$. Since the extension
  $K(\zeta_p)/K$ is unramified at $\ell$, the same is true for
  $\overline{\rho_{\id{p}}^-}$. The same argument proves that the
  residual representation is unramified at primes dividing $\beta$
  (since $\mot$ has monodromy matrix of order $2p$ at $1$). To prove
  the result for primes dividing $\gamma$, one considers the motive
  $\mot:=\HGM((\frac{1}{2},\frac{1}{2}),(\frac{1}{p},1)|t_0)$.

  The same proof works for all other motives $\mot_i^{\pm}$: 
  consider the hypergeometric motive obtained by adding $\frac{1}{p}$
  to one of the first two coordinates. This new motive will be
  unramified at primes dividing $\beta$ and also at primes dividing
  $\alpha$ when $i=2$.
\end{proof}

At last, we need a similar result for primes dividing $p$. Assume that $p \neq 2$.

\begin{theorem}
\label{thm:finite}
Let $\id{p}$ be a prime ideal of $K$ dividing $p$, such that
$\id{p} \nmid 2ABC$.  Then $\overline{\rho_{\id{p}}^{\pm}}$ arises
from a finite flat group scheme.
\end{theorem}

\begin{proof}
  The result for $i=1$ follows from \cite[Proposition 5]{MR885783}
  while for $i=2$ it follows from \cite[Proposition 1.15]{Darmon}. We
  focus on the case $i=4$ (the novelty of the present article). If
  $\id{p} \nmid \alpha \beta \gamma$, then $t_0$ reduces modulo
  $\id{p}$ to a point in $\PP^1\setminus\{0,1,\infty\}$, so the
  reduction of (a non-singular model of) Euler's
  curve~(\ref{eq:curve}) is non-singular (and so is the quadratic
  twist by $\gamma$). If $\id{p} \mid \alpha \gamma$ a similar result
  holds by Theorem~\ref{thm:hgm-specialization} (recall that $N=qr$ or
  $2qr$).

  The proof when $\id{p}\mid \beta$ depends on the theory of Mumford
  curves (as in Darmon's proof). Consider Euler's curve attached to
  $\mot_{4}^+$ at $t_0$, given with equation
  \[
    \bfC: y^{qr}=x^A(1-x)^B(1-t_0x)^C t_0^{-sr},
  \]
  for $A=q-sr$, $B=-q-sr$ and $C=q+rs$. Let $v=v_{\id{p}}(\beta)$ and
  write $t_0=1+u_0$, so that $u_0=\pi^{vp}\widetilde{u_0}$, for
  $\widetilde{u_0}$ a unit in $\Om_{\id{p}}^\times$. Consider the
  curve (in the variable $u$) with equation
  \begin{equation}
    \label{eq:mult}
        \bfC: y^{qr}=x^A(1-x)^B(1-x(u+1))^C t_0^{-sr},
  \end{equation}
  The stable model of its special fiber (see \cite{BW} and
  \cite{PV23}) consists of two curves $\bfC_1$ and $\bfC_2$
  intersecting at a point, given by the equations
  \[
    \bfC_1: y^{qr} = x^A(1-x)^{B+C} t_0^{-rs},
  \]
  and
  \[
    \bfC_2: y_2^{qr} = x_2^{B}(x_2-\widetilde{u_0})^C t_0^{-rs}.
  \]
  The second curve is obtained via the change of variables
  $1-x=\pi^{vp}x_2$ and $y_2=y$ (noting that $B+C=0$). Both curves
  have genus zero (since $B+C=0$) and their intersection points are
  defined over the extension $K_{\id{p}}(\sqrt[qr]{t_0})$, an
  unramified extension of $K_{\id{p}}$, so the reduction type of
  $\Jac(\bfC)$ is that of multiplicative reduction and $\bfC$ is a
  Mumford curve. The result follows from the fact that its $p$-torsion
  points are defined over the extension $K[\sqrt[p]{u_0}]$ (as in the
  proof of Proposition 1.15 in \cite{Darmon}). The
  proof for $\mot_{4}^-$ follows from a similar computation.

  The proof for $\mot_{3,s}^{\pm}$ is similar, if $\id{p}\nmid \beta$,
  the motive has good reduction. When $\id{p}\mid \beta$, the
  semistable model of Euler's curve with parameters
  $(\frac{1}{q},-\frac{1}{q}),(\frac{s}{q},-\frac{s}{q})$ is the
  intersection of the two genus $0$ curves
  \[
    \bfC_1: y^{q} = x^A(1-x)^{B+C} t_0^{-s},
  \]
  and
  \[
    \bfC_2: y_2^{q} = x_2^{B}(x_2-\widetilde{u_0})^C t_0^{-s}.
  \]
\end{proof}

\begin{remark}
  The motive attached to the parameters $(a/q,-a/q),(1,1)$ gives the
  same motive as the hyperelliptic curve $C_r^+(\alpha,\beta,\gamma)$
  studied in \cite{Darmon}, as proven in
  Corollary~\ref{coro:comparison-Darmon}.
\end{remark}

\section{Modularity}
\label{section:modularity}
Let us denote by $\mot_1^+$ the hypergeometric motive attached to the
parameters $(1,1),(1,1)$. It is not generic, corresponding to the
Eisenstein series of weight $2$ (see Proposition 6.3 of
\cite{GPV}). Theorem~\ref{thm:congruences} implies the following
``congruences'' (up to Galois conjugation and quadratic twists)
\[
\mot_1^+ \equiv \mot_2^+ \pmod{\id{r}}, \qquad \mot_1^+ \equiv \mot_3^+\pmod{\id{q}}, \qquad \mot_{4}^+ \equiv \mot_2^+ \pmod{\id{q}},
\]
for any prime ideal $\id{r} \mid r$ (respectively $\id{q} \mid q$).
There is a subtlety here: the motive $\mot_1^+$ is defined over $\Q$,
the motive $\mot_2^+$ is defined over $K=\Q(\zeta_r)^+$ and they
become congruent while restricted to $\Gal_F$, for $F=\Q(\zeta_r)$.
Since the extension $F/K$ is cyclic, the motive $\mot_2^+$ is
isomorphic (up to a quadratic twist) to $\mot_1^+$ restricted to
$\Gal_K$. Since modularity is preserved under quadratic twists, by
abuse of notation, we will ``assume'' that the congruences to hold
over $K$ (the same applies to the motives $\mot_i^-$).

Formula (\ref{eq:change-param}) gives the extra isomorphism
  \begin{equation}
    \label{eq:invertion}
\mot_{4}^+ \equiv \HGM\left(\left(\frac{1}{q},-\frac{1}{q}\right),(1,1)|t_0^{-1}\right) \pmod r.    
  \end{equation}
Similar congruences hold for the family $\mot_i^-$. General modularity
lifting theorems allow (with appropriate hypothesis) to propagate
modularity from one motive to another. When the residual
representations are irreducible, no hypothesis is needed, thanks to
general results of Kisin (\cite{MR2600871}). When the residual
representation is reducible the theory is not complete and extra
hypotheses appear (for example the main result of \cite{MR1793414}
requires the representations to be ordinary at the congruence
prime). In \cite{MR4467307} the author proves a stronger version for
abelian totally real extensions $K/\Q$ (which is our case). However, there
are still some technical hypothesis in his main result that we now
recall. Let $p$ be the residual characteristic, then the following must hold
\begin{enumerate}
\item The prime $p$ splits completely in $K$.
  
\item The global representation is either ordinary at all primes of
  $K$ dividing $p$ (Corollary 6.6.11 of \cite{MR4467307}) or
  
\item The global representation is irreducible while restricted at the
  decomposition subgroup $D_{\id{p}}$ for all primes $\id{p}$ of $K$
  dividing $p$ (Theorem 7.1.1 of loc. cit).
\end{enumerate}

It is not clear how to verify whether our motives satisfy one of the
last two conditions. A more general result (including mixed cases)
would imply modularity of all our motives. Still, we can prove some
partial results for our motives.

\begin{lemma}
  \label{lemma:irred-legendre}
  Let $p \ge 5$ be a prime number and let $E$ be Legendre's elliptic
  curve (\ref{eq:legendre}) specialized at any rational point
  $t_0 \in \Q\setminus\{0,1\}$. Then the residual representation
  $\overline{\rho_{E,p}}$ is irreducible.
\end{lemma}

\begin{proof}
  An elliptic curve whose residual Galois representation modulo $p$ is
  reducible has an isogeny of degree $p$.  Since the elliptic curve
  has full $2$-torsion, the result follows from Mazur's work
  (\cite{MR0482230}) when the curve is semistable. In general, a curve
  with full $2$-torsion and a degree $p$ isogeny gives a curve with an
  isogeny of degree $4p$, and there are no such rational elliptic
  curves when $p \ge 5$ by Kenku's result (see \cite{MR0616546}).
\end{proof}

\begin{lemma}
  \label{lemma:irred-K}
  Let $p\ge 5$ be a prime number, and let $F=\Q(\zeta_p)$. Then the
  Galois representation of the residual representation of Legendre's
  elliptic curve $E$ specialized at any rational point
  $t_0 \in \Q\setminus\{0,1\}$ restricted to $\Gal_F$ is irreducible.
\end{lemma}

\begin{proof}
%
  The previous lemma implies that the residual representation
  $\overline{\rho_{E,p}}$ is irreducible. By \cite[Lemma
  1.13]{MR4630566} the restriction to $\Gal_F$ is reducible if and
  only if the restriction to $\Gal_{\Q(\sqrt{p^\star})}$ is reducible,
  where $p^\star = \kro{-1}{p}p$. When the restriction is reducible,
  we are in what is usually called a \emph{bad dihedral} prime. Let
  $k$ be Serre's weight of the (irreducible) representation
  $\overline{\rho_{E,p}}$. If $p$ does not divide the denominator of
  $t_0$, the curve $E$ has either good or multiplicative reduction at
  $p$. When $p$ does divide the denominator, the same is true for the
  quadratic ramified twist of $E$ (twisting preserves
  irreducibility). So we can assume (up to a quadratic twist) that
  $k=2$ or $k=p+1$. On the other hand, by \cite[Lemma 1.14]{MR4630566}
  either $p=2k-3$ or $p=2k-1$. These conditions are not compatible
  when $p \ge 5$.
\end{proof}

\begin{theorem}
  If $r \mid \alpha\gamma$ then $\mot_2^+$ is modular.
\label{thm:mod-2+}
\end{theorem}
\begin{proof}
  The result if proven in \cite[Theorem 2.9]{Darmon}. Let $\id{r}$ be
  a prime ideal dividing $r$.  As noticed in \cite{Darmon}, the
  hypothesis $r \mid \alpha\gamma$ implies that the representation
  $\rho_{\id{r}}$ is ordinary at all primes dividing $r$, hence the
  result follows from \cite{MR1793414}.
\end{proof}
%
%
%
\begin{theorem}
  \label{thm:mod-2-}
  If $r \ge 5$, the motive $\mot_2^-$ is modular.
\end{theorem}
\begin{proof}
  The proof follows the argument given in \cite[Theorem 2.9]{Darmon}:
  let $\id{r}$ be a prime ideal of $K$ dividing $r$. Let
  $\rho_{\id{r}}^-$ be the $\id{r}$-adic representation attached to
  $\mot_2^-$. Its residual representation is isomorphic to the
  residual representation attached to the motive $\mot_1^-$ restricted
  to $\Gal_K$ (corresponding to the elliptic curve $E_{t_0}$ of
  (\ref{eq:legendre})). By the Shimura-Taniyama conjecture (proved in
  \cite{Wiles} and \cite{MR1839918}) $E_{t_0}$ is modular, i.e. there
  exists a weight $2$ modular form $f_{t_0}$ for the group
  $\Gamma_0(N)$ (for an appropriate $N$) whose $L$-series matches that
  of $E_{t_0}$. Let $F_{t_0}$ be the Hilbert modular form (of parallel
  weight $2$) corresponding to the base change of $f_{t_0}$ to $K$
  (whose existence is proven in \cite{MR2950769}). Then the
  representation $\rho_{\id{r}}^-$ is congruent modulo $\id{r}$ to the
  Galois representation attached to $F_{t_0}$. By
  Lemma~\ref{lemma:irred-K} the residual Galois representation
  $\overline{\rho_{E_{t_0},r}}$ restricted to $\Gal_K$ is absolutely
  irreducible, so by \cite{MR2600871} $\mot_2^-$ is modular.
\end{proof}
\begin{remark}
  In Darmon's article, there is an extra requirement (that
  $p \mid \alpha \beta$). This was a technical hypothesis that modularity
  lifting theorems had at the time (like ordinariness of the residual
  representation). These hypotheses have been removed since Darmon's
  article (see section \S4 of \cite{2205.15861}).
\end{remark}
\begin{theorem}
  If $q \mid \gamma$ then the motive $\mot_{3,s}^+$ is modular for all $s$.
\label{thm:mod-3+}
\end{theorem}
\begin{proof}
  Follows the same argument given in the proof of
  Theorem~\ref{thm:mod-2+} (see \cite[Theorem 2.9]{Darmon}).
\end{proof}
\begin{theorem}
\label{thm:mod-3-}
  If $q\ge5$, the motive $\mot_{3,s}^-$ is modular for all $s$.
\end{theorem}
\begin{proof}
  Follows the same argument of Theorem~\ref{thm:mod-2-}. See also
  \cite[Theorem 2.9]{Darmon}.
\end{proof}
\begin{theorem}
\label{thm:mod-4-+}
  If $q \mid \alpha$ and $r \mid \gamma$ then $\mot_{4}^+$ is
  modular.
\end{theorem}
\begin{proof}
  Let $\id{q}$ be a prime ideal of $K$ dividing $q$ and let
  $\rho_{\id{q}}$ denote the $\id{q}$-th Galois representation
  attached to the motive $\mot_{4}^+$. The reduction of
  $\rho_{\id{q}}$ is isomorphic to the reduction of the $\id{q}$-th
  representation $\widetilde{\rho}_{\id{q}}$ of the motive $\mot_2^+$
  restricted to $\Gal_K$ (there is an abuse of notation here, since
  the motive $\HGM_2^+$ is specialized at
  $t_0=-\frac{A\alpha^q}{C\gamma^r}$). If $\overline{\rho_{\id{q}}}$ is
  reducible, then the hypothesis $q \mid \alpha$ implies that the
  reduction is ordinary at all prime ideals dividing $q$, hence by
  \cite{MR1793414} the representation $\rho_{\id{q}}$ is modular.

  Otherwise, $\overline{\rho_{\id{q}}}$ is absolutely irreducible and
  by \cite{MR2600871} it is enough to prove that $\mot_2^+$
  (specialized at $t_0=-\frac{A\alpha^q}{C\gamma^r}$) is modular.  Let
  $\id{r}$ be the prime ideal of $\Q(\zeta_r)^+$ dividing $r$. The
  reduction of $\widetilde{\rho}_{\id{r}}$ is reducible, hence modular
  when $r \mid \gamma$ by Theorem~\ref{thm:mod-2+}.
\end{proof}
\begin{remark}
  In the previous result we do not expect the hypothesis
  $r \mid \gamma$ to be needed, as in general, the residual
  representation of $\widetilde{\rho}_{\id{r}}$ should be irreducible
  (if $r$ is large enough).
\end{remark}
%
%
\begin{theorem}
  If either $r \ge 5$ and $q \mid \alpha$ or $q \ge5$ and $r \mid \gamma$ then
  $\mot_{4}^-$ is modular.
  \label{theorem:modularity}
\end{theorem}

\begin{proof}
  Let $\id{q}$ be a prime ideal of $K$ dividing $q$ and let
  $\rho_{\id{q}}$ be the $\id{q}$-th Galois representation attached to
  $\mot_{4}^-$. If the residual representation
  $\overline{\rho_{\id{q}}}$ is reducible, then it is modular by
  \cite{MR1793414} (because $q \mid \alpha$). Otherwise, it is
  congruent to $\mot_2^-$ restricted to $\Gal_K$ (specialized at
  $t_0=-\frac{A\alpha^q}{C\gamma^r}$). The later motive is modular
  over $\Q(\zeta_r)^+$ by Theorem~\ref{thm:mod-2-} and also
  over $K$ by cyclic base change. Then $\mot_{4}^-$ is modular by
  \cite{MR2600871}. The proof in the second case hypothesis follows
  from a similar argument using (\ref{eq:invertion}).
\end{proof}

\begin{remark}
  As in the previous results, we do not expect the conditions
  $q \mid \alpha$ or $r \mid \gamma$ to be needed as long as the primes
  are large enough (so that the residual representations are absolutely irreducible).
\end{remark}
\section{Some special families}
\label{section:particular-cases}
Let us consider the particular families $(2,p,r)$ and $(3,p,r)$. In
these two cases we can prove modularity of the involved motives for
$r$ larger than an explicit small constant. The proof depends on the
rational cases studied in section~\ref{section:rational-HGM}.

\subsection{The family $(2,p,r)$} We can assume that $r \ge 5$, since
the case $(2,p,3)$ has already been studied in the literature. To a putative solution $(\alpha,\beta,\gamma)$ of the equation

\[
Ax^2+By^p+Cz^r=0,
\]
one can attach the hypergeometric motive
$\HGM_4^-:=\HGM\left(\left(\frac{1}{r},-\frac{1}{r}\right),\left(\frac{1}{4},-\frac{1}{4}\right)|t\right)$. Let $\mot_4^-$ denote the specialization of $\HGM_4^-$ at
$t_0:=-\frac{A\alpha^2}{C\gamma^r}$, twisted by the character
$\theta_\alpha$. If $\id{p}$ denotes a prime ideal of
$K=\Q(\zeta_r)^+$ dividing $p$, then the $\id{p}$-th residual
representation attached to $\mot$ is unramified outside $2r$.

\begin{theorem}
  Suppose that $r \nmid A$. If $r \ge 11$ or if $r=7$ and
  $2 \mid v_7(B)$, the motive $\mot_4^-$ is modular.
\label{thm:modularity-2pr-}
\end{theorem}
\begin{proof}
  Let $\id{r}$ denote the prime ideal of $K$ dividing $r$. The
  $\id{r}$-th Galois representation attached to $\mot$ is
  congruent to the quadratic twist by $\theta_\alpha$ of the motive
  $\HGM((\frac{1}{4},-\frac{1}{4}),(1,1)|t_0^{-1})$ of \S
  \ref{subsection:1/4}. It corresponds to the elliptic curve
\[
E_t:y^2+xy=x^3+\frac{t}{64}x.
\]
The curve $E_t$ has a rational $2$-torsion point, hence Kenku's result
(\cite{MR0616546}) implies that the residual image of $E_t$ is
absolutely irreducible for all primes $r >7$ (for all specializations
of the parameter). The case $r=7$ is also irreducible, since the curve
$X_0(14)$ has only two points corresponding to the $j$-invariants
$-3375$ and $16581375$. They correspond to the values
$t_0=\frac{64}{63}$ and $t_0=-\frac{1}{63}$, but note that
$63 = 9 \cdot 7$, which is not a square if $v_7(AB)=v_7(B)$ is even,
hence it does not come from any solution of the equation
$Ax^2+By^p+Cz^r=0$ with $r\ge 2$. Irreducibility of the residual
representation restricted to $\Gal_K$ follows from the same argument
used in Lemma~\ref{lemma:irred-K}: the curve has additive or
multiplicative reduction at $r$. The reduction of $E_{t_0^{-1}}$ at a
prime $r>3$ is either good (if $v_r(t_0^{-1}(t_0^{-1}-1))=0$ by
\eqref{eq:1/4-inv}) or multiplicative when
$v_r(t_0^{-1}(t_0^{-1}-1))>0$. When $v_r(t_0)>0$, since
$t_0^{-1}=-\frac{C\gamma^r}{A\alpha^2}$ and $r \nmid A$, the motive
$\mot_4^-$ has good reduction at $r$ by
Theorem~\ref{thm:hgm-specialization}.
\end{proof}
\begin{remark}
  It seems plausible that the residual image modulo $5$ at points
  coming from solutions is also irreducible, but we did not study
  this problem in detail (this is just a speculation based on some
  numerical experiments). Although this second motive could be used to run a
  multi-Frey approach, the problem is that the conductor exponent at
  $2$ is larger than the previous one.
\end{remark}

\subsection{The family $(3,p,r)$} Assume that $r \ge 5$, since the
case $(3,3,r)$ has already been studied in the literature.  Let
$\HGM_4^+:=\HGM((\frac{1}{r},-\frac{1}{r}),(\frac{1}{3},-\frac{1}{3})|t)$,
an hypergeometric motive also defined over $K:=\Q(\zeta_r)^+$. Its
monodromy matrices at the points $\{0,1,\infty\}$ have order
$\{3,\infty,r\}$ respectively.  Let $(\alpha,\beta,\gamma)$ be a solution to
\[
  Ax^3+By^p+Cz^r=0,
\]
and set $t_0:=-\frac{A\alpha^3}{C\gamma^r}$.

\begin{theorem}
\label{thm:modularity-3pr+}
  If $r \nmid A$, the motive $\HGM_4^+(t_0)$ is modular.
\end{theorem}
\begin{proof}
  Let $\id{r}$ be a prime ideal of $K$ dividing $r$.  The motive
  $\HGM_4^+(t_0)$ is congruent modulo $\id{r}$ to the hypergeometric motive
  $\HGM((\frac{1}{3},-\frac{1}{3}),(1,1)|t_0^{-1})$ corresponding (as
  described in \S \ref{section:p-p-3}) to the elliptic curve with
  equation
    \begin{equation}
      \label{eq:curve-3-q-r}
E_t:y^2+xy+\frac{t_0}{27}y=x^3.      
    \end{equation}
    The point $P=(0,0)$ belongs to $E$ and has order $3$. By Kenku's
    result \cite{MR0616546} (based in Mazur's result \cite{MR0482230})
    there are no rational elliptic curves with a cyclic isogeny of
    order $3p$ for $p \ge 11$. Furthermore, there are finitely many
    curves (up to quadratic twists) with an isogeny of order $15$ and
    $21$ (see \cite{MR0376533} page 79 and Table 4). 

    Any curve with an isogeny of order $15$ has the same $j$-invariant
    as a curve in the isogeny graph of curves of level $50.b$ (with
    LMFDB label $\lmfdbec{50}{b}{1}$, $\lmfdbec{50}{b}{2}$,
    $\lmfdbec{50}{b}{3}$ and $\lmfdbec{50}{b}{4}$), while any curve
    with an isogeny of order $21$ has the same $j$-invariant as a
    curve in the isogeny graph $162b$, with LMFDB labels
    $\lmfdbec{162}{b}{1}$, $\lmfdbec{162}{b}{2}$,
    $\lmfdbec{162}{b}{3}$ and $\lmfdbec{162}{b}{4}$. It is easy to
    verify that $j(E_t)$ does not match any of these 8 values for any
    rational value of $t$. In particular, the curve $E_t$ has
    absolutely irreducible residual image for $q \ge 5$. Its
    restriction to $\Gal_K$ is also absolutely irreducible (by the
    same argument given in Lemma~\ref{lemma:irred-K} and
    Theorem~\ref{thm:modularity-2pr-} under the assumption
    $r \nmid A$). The result then follows from \cite{MR2600871}.
\end{proof}

\begin{remark}
  One can obtain a bound for the conductor exponent of $\HGM_3^+(t)$ at
  primes dividing $3$ by studying the conductor exponent of the
  curve~\eqref{eq:curve-3-q-r} as done in
  Lemma~\ref{lemma:valuation-2} for the prime $2$.
\end{remark}

A similar result holds for the motive $\HGM_4^-(t)$ corresponding to
the parameters
$\left(\frac{1}{r},-\frac{1}{r}\right),\left(\frac{1}{6},-\frac{1}{6}\right)$
(this is a quadratic twist of the motive
$\HGM\left(\left(\frac{1}{2r},-\frac{1}{2r}\right),\left(\frac{1}{3},-\frac{1}{3}\right)|t\right)$
  as explained in \S \ref{section:quad-twist}).
\begin{theorem}
  \label{thm:modularity-3pr-}
  Suppose that $r \nmid A$. Then if $r \ge 11$ the motive
  $\HGM_4^-(t_0)$ is modular.
\end{theorem}
\begin{proof}
  Let $\id{r}$ be a prime ideal of $K$ dividing $r$. Then for any
  specialization of the parameter, the motive $\HGM_4^-(t)$ is
  congruent modulo $\id{r}$ to the restriction to $\Gal_K$ of motive
  $\HGM\left(\left(\frac{1}{6},-\frac{1}{6}\right),(1,1)|t\right)$. As studied in
  \S \ref{section:rational-HGM}, the later corresponds to the
  elliptic curve
\[
E_t:y^2+xy=x^3-\frac{t}{432}.
\]
The curve $E_t$ does not have any torsion point, hence proving that
its residual image modulo $r$ is irreducible seems more
challenging. By Mazur's result on isogenies of prime degrees
(\cite{MR0482230}) this is the case if $r$ does not belong to the set
$\{5, 7, 11, 17, 19, 37, 43, 67, 163\}$. The result holds when
$r \ge 11$ because for each of these values there are finitely many
$j$-invariants of elliptic curves with reducible image modulo $r$ (see
for example \cite{MR0376533}), namely:
  \begin{itemize}
  \item Any elliptic curve with a rational $11$-isogeny either has CM,
    or is a quadratic twist of one of the curves
    $\lmfdbec{121}{a}{1}$, $\lmfdbec{121}{a}{2}$ or
    $\lmfdbec{121}{b}{1}$, i.e. its $j$-invariant equals $-121$,
    $-32768$ or $-24729001$. Only the first $j$-invariant corresponds
    to the values $t_0=\frac{27}{11}$ or $t_0=\frac{-16}{11}$, whose
    denominator is a prime number (so does not come from a solution).
    
  \item Any elliptic curve with a rational $17$-isogeny is a quadratic
    twist of either the curve $\lmfdbec{14450}{b}{1}$ or
    $\lmfdbec{14450}{b}{2}$ (using the LMFDB label), with
    $j$-invariants $-\frac{882216989}{131072}$ and
    $-\frac{297756989}{2}$. None of them are $j$-invariants of $E_t$
    for a rational value of $t$.
    
  \item Any rational curves having a rational $19$, $37$, $43$, $67$
    or $163$ isogeny has CM. Then its $j$-invariant must be
    (respectively) $-884736$, $-12288000$, $-884736000$,
    $-147197952000$ and $-262537412640768000$. None of them belong to
    the family $E_t$.
  \end{itemize}
  Then modularity follows from Kisin's result.
\end{proof}
We do not know whether the last large image result holds for $r=5$ or
$r=7$ (but it is an interesting problem to study).

\section{Bounds at wilds primes}
\label{section:wild}
As mentioned in the introduction, it seems like a very hard problem to
determine the conductor exponents of an hypergeometric motive at a
wild prime $\id{p}$ dividing a rational prime $p$. However we can give
an explicit bound under the following assumption:
\begin{ass}
  \label{ass:square-free}
  The parameters $(a,b),(c,d)$ are generic and satisfy  the
following properties:
\begin{itemize}
\item $a+b$ and $c+d$ are integers,
  
\item if $p$ is a prime dividing $N$, then either $v_p(a)=v_p(b)= -1$
  and $v_p(c)=v_p(d)=0$ or vice-versa.
\end{itemize}
\end{ass}
Sometimes we will impose an extra hypothesis on the specialization
$t_0$ in order to lower the number of cases to consider, but the same
approach applies to the general case.

Recall that if $\rho:\Gal_K \to \GL_2(\overline{K_\lambda})$ is a
Galois representation and $\id{q}$ is a prime ideal whose residual
characteristic is prime to that of $\lambda$, then the $\id{q}$-th
valuation of the Artin conductor of $\rho$ at $\id{q}$ is computed by
\[
  \nn_{\id{q}}(\rho) = \nn_{\id{q},tame}(\rho) +
  \nn_{\id{q},wild}(\rho),
  \]
  where $\nn_{\id{q},tame}(\rho)$ is the codimension of the subspace
  fixed by the inertia group $I_{\id{q}}$ while
  $\nn_{\id{q},wild}(\rho)$, the \emph{Swan} conductor, is the sum
  of the codimensions over higher ramification groups.

  Our bound depends on the fact that two $\id{q}$-adic Galois
  representations which are residually isomorphic have the same Swan
  conductor at any prime ideal $\id{p}$ not dividing $\normid{q}$ (the
  norm of $\id{q}$).
  \begin{lemma}
    Under Assumption~\ref{ass:square-free}, if $\id{q}$ is a prime
    ideal of $K$ dividing $N$ with residual characteristic $q$ and
    $t_0 \in \QQ$ then there exists $s \in \FF_q^\times$ such that
\[
\nn_{\id{q},wild}(\hgms)= 
\begin{cases}
  \nn_{\id{q},wild}\left(\HGM\left(\left(\frac{s}{q},-\frac{s}{q}\right),\left(1,1\right)|t_0\right)\right) & \text{if} \ v_{q}(a)= -1,\\
  \nn_{\id{q},wild}\left(\HGM\left(\left(\frac{s}{q},-\frac{s}{q}\right),\left(1,1\right)|\frac{1}{t_0}\right)\right) & \text{if} \ v_{q}(c)= -1.
\end{cases}
\]
\label{lemma:reduction-cong}
\end{lemma}
\begin{proof}
  The isomorphism (\ref{eq:change-param}) implies that the second case
  follows from the first one.  Since $a+b$ and $c+d$ are integers we
  can assume that $b=-a$ and $d=-c$.

  Let $p$ be a rational prime different from $q$ such that
  $v_p(a)=-1$. Then we can write $a=\frac{r}{p}+\frac{\alpha}{\beta}$
  with $v_p(\beta)=0$ and $v_q(\beta)=1$. Let $\id{p}$ be a prime
  ideal of $K$ dividing $p$. By Theorem~\ref{thm:congruences}, we have
  a congruence between $\hgms$ and
  $\HGM((\frac{\alpha}{\beta},-\frac{\alpha}{\beta}),(c,d)|t_0)$ over
  $K(\zeta_p)$ modulo a prime ideal of $K(\zeta_p)$ dividing $\id{p}$,
  so both motives have the same Swan conductor at $\id{q}$. Since the
  extension $K(\zeta_p)/K$ is unramified at $\id{q}$, they both have
  the same Swan conductor over $K$. Applying the same procedure for
  each rational prime dividing $\beta$ different from $q$ it follows that
  \[
    \nn_{\id{q},wild}(\hgms)=
    \nn_{\id{q},wild}\left(\HGM\left(\left(\frac{s}{q},-\frac{s}{q}\right),(c,d)|t_0\right)\right).
  \]
    Since the denominator of $c$ is prime to $q$, a similar
    argument proves the result.
\end{proof}
\begin{remark}
  By \eqref{eq:galois-action}, the motives
  $\HGM\left(\left(\frac{1}{q},-\frac{1}{q}\right),(1,1)|t_0\right)$
  and
  $\HGM\left(\left(\frac{s}{q},-\frac{s}{q}\right),(1,1)|t_0\right)$
  are Galois conjugate and both appear in the same superelliptic
  curve (which is defined over $\Q$), so their conductor exponent at a
  prime ideal $\id{p}$ are the same.
\end{remark}

The lemma implies that under Assumption~\ref{ass:square-free}, the
Swan conductor of our motive is the same as the Swan conductor of the
motive
$\HGM\left(\left(\frac{1}{q},-\frac{1}{q}\right),(1,1)|t_0\right)$ (or
its specialization at $t_0^{-1}$), which appears in the Jacobian of an
hyperelliptic curve (as proved in
Appendix~\ref{appendix:hyperelliptic}), so one can use the theory of
cluster pictures to compute its Swan conductor at odd primes (as done
in \cite{2503.21568}).

Computing the local type of an hyperelliptic curve at $2$ is also
quite challenging, but for $q=2$ the motive is actually an elliptic
curve whose conductor was computed in
Lemma~\ref{lemma:valuation-2}.  This approach allows us
to give upper bounds for the conductor exponents at odd wild primes
for the motives $\mot_2^{\pm}$ and $\mot_4^{\pm}$ and at primes
dividing $2$ for all motives $\mot_i^-$.

A prime ideal $\id{q}$ is called \emph{wild} if
$\nn_{\id{q},wild}(\rho) \neq 0$. In the case of $2$-dimensional
representations we have that
  \begin{equation}
    \label{eq:tame-bound}
\nn_{\id{q},tame}(\rho) \in
\begin{cases}
  \{1,2\} & \text{ if }\id{q} \text{ is wild},\\
  \{0,1,2\} & \text{ otherwise.}\\
\end{cases}
  \end{equation}
  The wild primes of an hypergeometric motive always divides $N$.
\begin{remark}
  If $\id{q}$ is an odd prime (i.e. $\id{q} \nmid 2$) then the Swan
  conductor of $\rho$ at $\id{q}$ is invariant under quadratic
  twists. This fact plays a crucial role in our computations.
  \label{rem:swan-twists}
\end{remark}

\subsection{Primes dividing $2$}

Let $\id{q}$ be a prime ideal of $K$ dividing $2$.

%
\begin{theorem}
  \label{thm:bound-at-2}
  Let $(a,b),(c,d)$ be parameters satisfying
  Assumption~\ref{ass:square-free}, with $v_2(a)=-1$. Let $s$ be the
  valuation at $\id{q}$ of the conductor of $\hgms$. Then:
  \begin{itemize}
  \item If $v_2(t_0) \ge 0$, 
    \[
     s =
     \begin{cases}
       5 & \text{ if }t_0 \equiv 3 \pmod 4,\\
       4 & \text{ if }t_0 \equiv 1 \pmod{4}, \\
       0,1,2 & \text{ if }v_2(t_0)\ge 5.
     \end{cases}
   \]
   
 \item If $v_2(t_0) <0$,
   \[
     s=
     \begin{cases}
       6 & \text{ if } 2\nmid v_2(t_0),\\
       4 & \text{ if } 2 \mid v_2(t_0) \text{ and }t_0/2^{v_2(t_0)} \equiv 3 \pmod 4,\\
       0,1,2 & \text{ if } 2 \mid v_2(t_0)<-4 \text{ and }t_0/2^{v_2(t_0)} \equiv 1 \pmod 4.
     \end{cases}
\]
   \end{itemize}
\end{theorem}
\begin{proof}
%
  By the proof of Lemma~\ref{lemma:reduction-cong}, our motive is
  related under a finite number of congruences (modulo odd prime
  ideals) and base extensions (unramified at $2$) to the motive
  $\HGM((\frac{1}{2},\frac{1}{2})(1,1)|t_0)$.
  
  Let $k$ be the conductor exponent at $2$ of $E_{t_0}$ given in
  Lemma~\ref{lemma:valuation-2}. If $t_0\equiv 3 \pmod 4$, then
  $k=5$. Looking at Table 1 of \cite{2203.07787} we see that the
  Weil-Deligne type of $E_{t_0}$ is supercuspidal, and the extension
  where the curve attains good reduction has ramification degree
  $8$. Equivalently, our local type corresponds to the induction from
  $\QQ_2(\sqrt{-1})/\QQ_2$ of a character of conductor $\sqrt{-1}^3$
  and order $4$ (see \cite[Corollary 4.1]{MR3056552} and \cite[Theorem
  2.10]{MR4269428}). In particular, the Swan part is $3$ and its tame
  part $2$. By Lemma~\ref{lemma:reduction-cong} the Swan conductor of
  $\hgms$ is also $3$. More can be said on the tame conductor: since
  all congruences involve odd primes, and our base extensions are
  unramified at $2$, the local type of the residual representations
  involved in the congruences is also supercuspidal, hence their
  residual tame conductor is $2$, so $s=5$ as claimed. The other cases
  follow from a similar argument, via the following observations:
  \begin{itemize}
  \item If $k=4$, it follows from Table 1 of \cite{2203.07787} that
    the Weil-Deligne type of $E_{t_0}$ is either:
    \begin{enumerate}
    \item A twist (by the quadratic
    character $\varepsilon_{-1}$ corresponding to
    $\Q_2(\sqrt{-1})/\Q_2$) of the Steinberg representation.
    
  \item A principal series, twist by $\varepsilon_{-1}$ of an
    unramified representation.
    
  \item Supercuspidal, induced from the unramified extension
    $\QQ_2(\sqrt{5})/\QQ_2$ of a character of conductor $2^2$ and
    order $6$ (which is the same as taking the twist by
    $\varepsilon_{-1}$ of the induction of the character of order $3$
    and conductor $2$ as described in Table 1 of
    \cite{2203.07787}). 
    
  \item Exceptional supercuspidal representation, twist by
    $\varepsilon_{-1}$ of the exceptional supercuspidal of conductor
    $2^3$.
    \end{enumerate}
    In all cases, the residual tame conductor exponent is $2$, hence
    $\hgms$ has the same conductor at $\id{q}$ as $E_{t_0}$.
   \item If $k=1$, the Swan conductor is $0$, hence $\hgms$ has
     conductor exponent $s=0,1,2$ at $\id{q}$.     
   \item If $k=6$, the Weil-Deligne type of $E_{t_0}$ is either:
     \begin{enumerate}
     \item A quadratic twist by $\varepsilon_{\pm2}$ (the character
       attached to the extension $\QQ_2(\sqrt{\pm 2})/\QQ_2$) of a Steinberg
       representation.
       
     \item A principal series, twist by $\varepsilon_{\pm2}$ of an
       unramified one.
       
     \item Supercuspidal, induced from either: a character of
       $\QQ_2(\sqrt{5})$ of conductor $2^3$ and order $6$, a character
       of $\QQ_2(\sqrt{-1})$ or $\QQ_2(\sqrt{-5})$ of conductor $2^2$
       (of valuation $4$) and order $4$.
       
     \item Exceptional supercuspidal.
     \end{enumerate}
     In all cases, the residual tame conductor exponent is $2$.
   \end{itemize}
\end{proof}

\begin{corollary}
  \label{coro:bound-2}
Let $\id{q}$ be a prime ideal of $K$ dividing $2$ and suppose $2 \nmid AC$. Then, for $p,q,r \geq 5$, $\nn_{\id{q}}(\mot_i^-) \leq 6$ for all $i$.
\end{corollary}

\begin{proof}
  Recall that $t_0=-\frac{A\alpha^q}{C\gamma^r}$. The assumption
  $2 \nmid AC$ implies that either $v_2(t_0)=0$, $v_2(t_0)\ge 5$ or
  $v_2(t_0)<-4$. In all cases, the last theorem implies that the
  conductor exponent of the motive $\hgms$ is at most $6$. The
  quadratic character $\theta_\gamma$ has conductor valuation at most
  $3$ at $2$, hence the quadratic twist $\mot_i^-$ also has conductor
  exponent at most $6$ at $\id{q}$.
\end{proof}

\subsection{Odd primes}
Let $\id{q}$ be an odd prime of $K$ dividing $N$ of residual characteristic $q$.

\begin{theorem}
\label{thm:bound-odd-wild}
  Let $(a,b),(c,d)$ be parameters satisfying
  Assumption~\ref{ass:square-free}. Let $s$ be its conductor exponent
  at $\id{q}$. Then
  \[
    s \le
    \begin{cases}
      \frac{q+5}{2} & \text{ if }v_{\id{q}}(t_0(t_0-1))=1,\\
      q+2 & \text{ if } q\nmid v_q(t_0)<0,\\
      3 & \text{ otherwise.}\\
    \end{cases}
  \]
\end{theorem}
\begin{proof}
  Recall that if $\bfC$ is Euler's curve attached to the parameters
  $(a,b),(c,d)$, then the motive $\hgms$ is defined by
  $\M_{\bfC} \otimes
  \JacMot((-a,-b,c,d),(c-b,d-a))^{-1}(-1)^{d-b}$. When $a+b$ and $c+d$
  are integers, the Jacobi motive is at most a Tate twist (hence it
  does not affect the action of inertia groups).  Under
  Assumption~\ref{ass:square-free}, Lemma~\ref{lemma:quad-char}
  implies that the character $(-1)^{d-b}$ is unramified outside $2$,
  so the conductor of the motive matches the conductor of (a
  $2$-dimensional part of) Euler's curve.

  By Lemma~\ref{lemma:reduction-cong} it is enough to understand the
  Swan conductor of 
  $\HGM\left(\left(\frac{1}{q},-\frac{1}{q}\right),(1,1)|t_0^{\pm 1}\right)$, whose Euler's
  curve is isomorphic (by Theorem~\ref{coro:comparison-Darmon}) to an
  hyperelliptic curve. Its Swan conductor is given in
  \cite{2503.21568} (see formula (36)): it is either equal to $0$, or
  it equals $1$ when $v_{\id{q}}(t_0(t_0-1))\in\{0,2\}$,
  $\frac{q+1}{2}$ if $v_{\id{q}}(t_0(t_0-1)) =1$ and $q$ if
  $v_q(t_0)<0$ and $q \nmid v_q(t_0)$.
\end{proof}

\begin{corollary}
\label{coro:odd-bound}
Let $\id{q}$ be an odd prime dividing $N$ with residual characteristic
$q$. Let $s_i$ be the conductor exponent of the motive $\mot_i^{\pm}$,
for $i=2,4$. Then
\[
s_i \le
\begin{cases}
  3 & \text{ if } \id{q} \nmid ABC,\\
  q+2 & \text{ otherwise.}
\end{cases}
\]
\end{corollary}

\section{Eliminating modular forms using HGMs}
\label{section:elimination}
Let us focus in the case of general exponents $(q,p,r)$.  Suppose that
the first three steps of the modular method succeed, so we have
attached to a putative primitive solution $(\alpha,\beta,\gamma)$ of
equation (\ref{eq:genFermat}) a Hilbert newform
$g \in S_2(\Gamma_0(\id{n}))$ over a totally real base field $K$
(contained in $\Q(\zeta_N)$), where the level $\id{n}$ is only
divisible by primes dividing $ABCpq$. The form $g$ satisfies that it
is congruent modulo $\id{p}$ (a prime ideal of $K$ dividing the
exponent $p$) to $\hgms$ while restricted to $F$.

After computing the space $S_2(\Gamma_0(\id{n}))$ and its newforms, we
aim to discard newforms not related to the solution
$(\alpha,\beta,\gamma)$ (at least for large values of the prime
$r$). If we can discard them all then no solution can exist.

Let us explain an algorithm based on an idea of Mazur, which in
practice allows to discard most newforms when the exponent $p$ is
larger than a computable constant. Keeping the previous notation, let $N$ be
the least common multiple of the denominators of $a,b,c,d$ and let
$F=\Q(\zeta_N)$. Let $g$ be any newform in $S_2(\Gamma_0(\id{n}))$,
and denote by the same letter its base change to $F$. For $\id{l}$ a
prime ideal of $F$, denote by $a_{\id{l}}(g)$ the $\id{l}$-th Fourier
coefficient of $g$.

The algorithm consists on studying solutions to (\ref{eq:genFermat})
modulo $\ell$ (for $\ell$ a rational prime not dividing
$ABCpqr$), in order to get information on the Fourier coefficient
$a_{\id{l}}(F)$ for primes $\id{l}$ of $F$ dividing $\ell$. More
concretely, let
\[
  S_\ell = \{(\tilde{\alpha},\tilde{\beta},\tilde{\gamma}) \in
  \F_\ell^3\setminus\{(0,0,0)\} \; : \;
  A\tilde{\alpha}^q+B\tilde{\beta}^p+C\tilde{\gamma}^r=0\}.
\]
Then any primitive solution $(\alpha,\beta,\gamma)$ of
(\ref{eq:genFermat}) reduces modulo $\ell$ to an element in $S_\ell$.
In practice, since the prime $p$ is going to be larger than $\ell$,
raising to the $p$-th power is a bijection of $\F_\ell$, so we can
parametrize $S_\ell$ by elements
$(\tilde{\alpha},\tilde{\gamma}) \in \FF_\ell^2\setminus \{(0,0)\}$,
which determine $\tilde{\beta}$ uniquely. Denote by $S_\ell^\times$
the subset of $S_\ell$ made of elements where none of their entries
are zero. Consider the following three distinct cases:
\begin{enumerate}
\item The value $\tilde{\alpha}\tilde{\beta}\tilde{\gamma} \neq 0$ (i.e. $(\tilde{\alpha},\tilde{\beta},\tilde{\gamma}) \in S_\ell^\times$).
\item Either $\tilde{\alpha} = 0$ or $\tilde{\gamma}=0$, but $\tilde{\beta} \neq 0$.
\item The number $\tilde{\beta}$ equals $0$.
\end{enumerate}

In the first case, we set
$t_0=-\frac{A\tilde{\alpha}^q}{C\tilde{\gamma}^r}$, and compute the
value $H_{\id{l}}((a,b),(c,d)|t_0)$ (using
Definition~\ref{defi:finite-hgs}), whose value depends only $t_0$
modulo $\id{l}$. By part (2) of Theorem~\ref{thm:hgm-specialization}
the following congruence holds
\begin{equation}
  \label{eq:Mazur}
a_{\id{l}}(F) \equiv H_{\id{l}}((a,b),(c,d)|t_0) \pmod{\id{p}},  
\end{equation}
where $\id{p}$ is a prime ideal dividing $p$ in the composition of the
coefficient field of $g$ and $F$. 

\vspace{2pt}

\noindent{\bf Assumption 1:} Suppose that both sides of (\ref{eq:Mazur})
are different for all values $t_0$ obtained from elements of $S_\ell^\times$.

\vspace{2pt}

Then for each element $t_0$, the difference between the left and the
right hand side of (\ref{eq:Mazur}) is non-zero and $p$ must divide
(the norm) of their difference. This gives for each value of $t_0$ a
finite list of possibilities for $p$, so we can take their union and
get a bound for $p$ in case (1).

Suppose then that we are in case (2) with $\tilde{\alpha}=0$ (so
$\ell \mid \alpha$). Since $\ell \nmid A$, $v_{\ell}(A \alpha^q)$ is
divisible by $q$. The condition $\ell \nmid ABCpqr$ implies that
$\ell$ is not a wild prime so part (3) of
Theorem~\ref{thm:hgm-specialization} implies that the image of inertia
at $\id{l}$ is trivial (i.e. $\hgms$ is unramified at
$\id{l}$). Write the parameter $t_0$ in the form
$t_0 = \ell^{qv_{\ell}(\alpha)} \tilde{t_0}$, where
$\ell \nmid \tilde{t_0}$.
Then, by \cite[Theorem A.1]{GPV}, the trace of
$\Frob_{\id{l}}$ for $\hgms$ equals
\begin{multline}
  \label{eq:mazur-2}
  -\chi_{\id{l}}(-1)^{(d-b)(\norm(\id{l})-1)}{\bf J}((-a,-b,c,d),(c-b,d-a))(\id{l})^{-1} \cdot \\
  \cdot \left(\chi_{\id{l}}(\tilde{t_0})^{dN}J(\chi_{\id{l}}(\tilde{t_0})^{(d-b)N},\chi_{\id{l}}(\tilde{t_0})^{(b-c)N})+\chi_{\id{l}}(\tilde{t_0})(-1)^{(b-c)N}\chi_{\id{l}}(\tilde{t_0})^{cN}J(\chi_{\id{l}}(\tilde{t_0})^{(d-c)N},\chi_{\id{l}}(\tilde{t_0})^{(a-d)N})\right),
\end{multline}
where $J(\psi,\chi)$ is the usual Jacobi sum.

\vspace{2pt}
\noindent{\bf Assumption 2:} Suppose that $a_{\id{l}}(g)$ does not
match \eqref{eq:mazur-2} for any value $\widetilde{t_0} \in \FF_\ell^\times$.

\vspace{2pt}

Then once again, for each $t_0$ of type $(2)$ the difference between
\eqref{eq:mazur-2} and $a_{\id{l}}(g)$ is non-zero and $p$ must
divide its norm.

When $\tilde{\gamma}=0$ the strategy follows analogously: replacing
(using \cite[Theorem A.2]{GPV}) \eqref{eq:mazur-2} by
\begin{multline}
  \label{eq:mazur-2-b}
-\chi_{\id{l}}(-1)^{(d-b)(\norm(\id{l})-1)}{\bf J}((-a,-b,c,d),(c-b,d-a))(\id{l})^{-1} \cdot \\
  \cdot \left(\chi_{\id{l}}(\tilde{t_0})^{bN}J(\chi_{\id{l}}(\tilde{t_0})^{(d-b)N},\chi_{\id{l}}(\tilde{t_0})^{(a-d)N})+\chi_{\id{l}}(\tilde{t_0})(-1)^{(a-d)N}\chi_{\id{l}}(\tilde{t_0})^{aN}J(\chi_{\id{l}}(\tilde{t_0})^{(a-b)N},\chi_{\id{l}}(\tilde{t_0})^{(b-c)N})\right).
\end{multline}

Case $(3)$ (when $\id{l} \mid {\beta}$) lies in the so called
``lowering the level'' situation: by construction the hypergeometric
motive has multiplicative reduction at primes dividing $\beta$ (which
do not divide $B$), but $\id{l}$ does not divide $\id{n}$. Then the
well known lowering the level condition must hold, namely
\begin{equation}
  \label{eq:mazur-3}
a_{\id{l}}(g) \equiv \pm (\norm(\id{l})+1) \pmod{\id{p}}.  
\end{equation}
Note that $a_{\id{l}}(g)$ cannot be equal to $\pm (\norm(\id{l})+1)$
for $\ell$ large enough (as it contradicts Weil's bound on the number
of points of a curve over a finite field), so there are finitely many
possibilities for $p$.
  
In each case both sides of the expected congruence can be computed. If
the values happen to be different (i.e. Assumption 1, Assumption 2 and
the analogous assumption for $\tilde{\gamma}=0$ are true), we get a
finite list of candidates for the prime $p$ dividing the norm of their
difference. In practice, one varies $\ell$ over a small set of primes
and take the greatest common divisor of the different bounds to reduce
the possible values of $p$. This method is very powerful, and succeeds
to discard most newforms of $S_2(\Gamma_0(\id{n}))$. Sometimes it is
the case that no newforms (except for example some with complex
multiplication) pass the test, proving non-existence of
solutions to (\ref{eq:genFermat}) for all primes $p$ but a few small
ones. However in many instances, there are a few newforms which
systematically pass the test and other ideas are needed to discard
them.

\newpage

\appendix

\section{Motives coming from hyperelliptic curves}
\label{appendix:hyperelliptic}
\begin{center} by Ariel Pacetti and Fernando Rodriguez-Villegas \end{center}
\smallskip

The results of the present appendix follow the ideas of
\cite{MR1138583}. Let $N>1$ be an odd positive integer, $\zeta_N$ an
$N$-th primitive root of unity and $F=\Q(\zeta_N)$ the cyclotomic
field.  For $a \in \Q$, $a \neq \pm 2$, define the hyperelliptic curve
\begin{equation}
  \label{eq:curve1}
  \C_N':y^2=x^{2N}+ax^N+1. 
\end{equation}
The curve $\C_N'$ has two involutions, the canonical one
$\tau:\C_N' \to \C_N'$ given by $\tau(x,y)=(x,-y)$ and a second
involution $\iota_N:\C_N'\to \C_N'$ given by
$\iota_N(x,y)=\left(\frac{1}{x},\frac{y}{x^N}\right)$. A simple
computation proves that both involutions commute, so the group
$\Z/2 \times \Z/2$ is a subgroup of the automorphisms of $\C_N'$.

Let $g(x)$ denote the monic polynomial whose zeros are all the numbers
$\xi + \xi^{-1}$, for $\xi \in \overline{\Q}\setminus\{-1\}$
satisfying $\xi^N=-1$ (which matches the minimal polynomial of
$-(\zeta_N+\zeta_N^{-1})$ when $N$ is odd).  Let $\D_N$ denote the
hyperelliptic curve
\begin{equation}
  \label{eq:hyperell}
  \D_N:y^2=(x+2)(xg(x^2-2)+a).
\end{equation}

\begin{lemma}
The quotient of $\C_N'$ by the involution $\iota_N$ is isomorphic to $\D_N$.      
\end{lemma}
\begin{proof}
  See Proposition 3 of \cite{MR1138583}. The idea is that
  $\xi = x+\frac{1}{x}$ and $\eta = \frac{y(1+1/x)}{x^{(N-1)/2}}$
  generate the field of functions fixed by the involution, and the
  given equation is the relation they both satisfy (setting $y=\eta$
  and $x=\xi$).
\end{proof}

\begin{lemma}
  The quotient of $\C_N'$ by the involution $\iota \circ \tau$ is given by the equation
  \begin{equation}
    \label{eq:second-quotient}
   \D_N': y^2=(x-2)(xg(x^2-2)+a).
  \end{equation}
\end{lemma}
\begin{proof}
  Mimics the previous one, noting that the field of invariant
  functions is now generated by the functions $\xi=x+\frac{1}{x}$ and
  $\eta = \frac{y(1-1/x)}{x^{(N-1)/2}}$. The two variables satisfy the
  stated relation (setting $y=\eta$ and $x=\xi$).
\end{proof}

Since the quotient of $\C_N'$ by $\langle \tau, \iota_N\rangle$ has
genus $0$, up to isogeny
\[
\Jac(\C_N') \sim \Jac(\D_N) \times \Jac(\D_N').
\]

If $d \mid N$, there is a natural map 
$\pi_d:\C_N' \to \C_d'$ sending $(x,y) \to (x^{N/d},y)$. 
 It is easy to
verify that the following diagram is commutative
\[
\xymatrix{
\C_N'\ar[d]_{\pi_d} \ar[r]^{\iota_N} & \C_N'\ar[d]^{\pi_d}\\
\C_d' \ar[r]^{\iota_d} & \C_d'
}
\]
so the map $\pi_d$ also induces a map $\pi_d:\D_N \to \D_d$. The
pullback under $\pi_d$ of the Jacobian of $\D_d$ is a subvariety of
$\Jac(\D_N)$ (belonging to its \emph{old} part). Define the \emph{new}
part of $\Jac(\D_N)$ to be a complement to the contribution from all
of its old parts. Clearly
$$\Jac(\C_N')^{\text{new}} \sim \Jac(\D_N)^{\text{new}} \times
\Jac(\D_N')^{\text{new}}.$$

\begin{lemma}
  The dimension of $\Jac(\D_N)^{\text{new}}$ equals $\frac{\phi(N)}{2}$.
\label{lemma:app-dimension}
\end{lemma}
\begin{proof}
  The polynomial $g(x)$ has degree $\frac{N-1}{2}$, so the right hand
  side of (\ref{eq:hyperell}) has degree $N+1$ and the genus of $\D_N$
  equals $\frac{N-1}{2}$. If $N$ is a prime number, then the result
  follows. Otherwise, by an inductive argument, we can assume that the
  result holds for all divisors of $N$, hence
  \[
  \frac{N-1}{2} = g(\D_N) = \dim(\Jac(\D_N)^{\text{new}}) + \sum_{\stackrel{d\mid N}{1 < d < N}}\frac{\phi(d)}{2},
    \]
and the result follows from M\"obius inversion formula.
\end{proof}
An analogue result holds for $\Jac(\D_N')^{\text{new}}$.  The group
$\mubb_N$ of $N$-th roots of unity acts on the curve $\C_N'$ via
$\zeta_N\cdot(x,y)=(\zeta_N x,y)$ (this endomorphism is defined over
$F$). The action does not commute with $\iota_N$, but satisfies the
relation
\[
\zeta_N \circ \iota = \iota \circ \zeta_N^{-1}.
\]
Then over $K=\Q(\zeta_N)^+$, the ring $\ZZ[\zeta_N+\zeta_N^{-1}]$ acts
on $\Jac(\D_N)$. 
\begin{theorem}
  \label{thm:hyperelliptic}
  For any $t_0 \in \Q$, $t_0 \neq 0,1$,
  $\HGM((\frac{1}{N},-\frac{1}{N}),(1,1)|t_0)$ is up to a quadratic twist by $F/K$
  part of the new part of the hyperelliptic curve $\D_N$ for
  $a=2(1-2t_0)$.
\end{theorem}

\begin{proof}
  Since $N$ is odd, the motive $\HGM((\frac{1}{N},-\frac{1}{N}),(1,1)|t_0)$ is
  realized in the new part of the Jacobian of (the desingularization
  of) Euler's curve
  \begin{equation}
    \label{eq:euler}
    \C_N:y^N=x(1-x)^{N-1}(1-t_0x).
  \end{equation}
  A trivial change of variables translates this equation into
  \[
(1-x)y^N=x(1-t_0x).
\]
Then the function field of $\Q(\C_N)$ is a quadratic extension of
$\Q(y)$. Computing the discriminant (with respect to the variable $x$) of the last
equation, we get the alternate definition
\[
\C_N':z^2=y^{2N}+2(1-2t_0)y^N+1,
\]
matching the curve~(\ref{eq:curve1}) (changing variables $z=y$, $y=x$)
for the value $a=2(1-2t_0)$. The result follows from the previous
considerations.
\end{proof}

\begin{corollary}
  Let $p$ be a prime number and let $t_0 \in \Q\setminus \{0,1\}$. The
  hypergeometric motive $\HGM((\frac{1}{p},-\frac{1}{p}),(1,1)|t_0)$ matches the
  hypergeometric curve denoted by $C_p^+$ in \cite{Darmon}.
\label{coro:comparison-Darmon}
\end{corollary}

\begin{remark}
  Theorem~\ref{thm:hgm-specialization} states a precise formula for
  Frobenii elements of $\HGM((\frac{1}{p},-\frac{1}{p}),(1,1)|t_0)$
  over $F$ (not over $K$). Then a priori the numerical value
  $H_{\id{q}}(\frac{1}{p},-\frac{1}{p}),(1,1)|t_0)$ of might differ
  from Darmon's curve by the quadratic twist corresponding to the
  extension $F/K$.
\end{remark}

\begin{proof}
  The curve $C_p^+$ matches the hypergeometric curve $\D_N$ of
  (\ref{eq:hyperell}).
\end{proof}

The picture for $N$ even is more interesting. In this case the curve $\C_N'$
has an extra involution $\sigma:\C_N'\to \C_N'$ given by
$\sigma(x,y)=(-x,y)$. The three involutions commute with each other,
providing an action of the group $\Z/2 \times \Z/2 \times \Z/2$ on
$\C_N'$. The following result is elementary.

\begin{lemma}
  The quotient of $\C_N'$ by the involution $\sigma$ is given by
  \[
\C_{N/2}':y^2=x^{N}+ax^{N/2}+1.
    \]
\end{lemma}

\begin{lemma}
  The quotient of $\C_N'$ by the involution $\sigma \circ \tau$ is given by
  \[
\D_N:y^2=x(x^{N}+ax^{N/2}+1).
    \]
\end{lemma}

\begin{proof}
  The involution $\sigma \circ \tau$ send $(x,y) \to (-x,-y)$. The
  field of functions on $\C_N'$ fixed by it is generated by
  $\xi = x^2$ and $\eta = xy$, which satisfy the stated relation
\[
\eta^2=\xi(\xi^N+a\xi^{N/2}+1).
\]
\end{proof}
Since $\tau \in \langle \sigma \circ \tau, \sigma\rangle$, and the
quotient of $\C_N'$ by $\tau$ has genus zero, the varieties
$\Jac(\C_{N/2}')$ (of dimension $\frac{N}{2}-1$) and $\Jac(\D_N)$ (of
dimension $\frac{N}{2}$), seen as subvarieties of $\Jac(\C_N')$ under the
pullback of the quotient map, have finite intersection. Since the sum
of their dimensions match the dimension of $\Jac(\C_N')$, it follows
that
\begin{equation}
  \label{eq:decomposition-N-even}
  \Jac(\C_N') \sim \Jac(\C_{N/2}') \times \Jac(\D_N).
\end{equation}
Note that the \emph{old} part of $\Jac(\D_N)$ comes from divisors $d$
satisfying that $N/d$ is odd (the other divisors contribute to
$\Jac(\C_{N/2}')$). Explicitly, if $d \mid N$ and $N/d$ is odd, the
map
\[
\pi_d(x,y) = (x^{\frac{N}{d}},x^{\left(\frac{N}{d}-1\right)\frac{1}{2}}y),
\]
maps the curve $\D_N$ to the curve $\D_d$.
\begin{lemma}
  The new part of $\Jac(\D_N)$ has dimension $\phi(N)$.
\end{lemma}
\begin{proof}
  Follows from an easy computation as the one done in
  Lemma~\ref{lemma:app-dimension}, using that the old
  contribution comes from divisors $d$ satisfying that $N/d$ is odd.
\end{proof}
The action of the group of $N$-th roots of unity on $\C_N'$ commutes with
$\sigma \circ \tau$, hence it induces an action on $\D_N$ given
explicitly by
\[
\zeta_N \cdot(x,y)=(\zeta_N^2x,\zeta_Ny).
\]
In particular, the compatible systems of Galois representations
attached to $\Jac(\D_N)^{\text{new}}$ (viewed as $\Z[\zeta_N]$-module)
are the ones corresponding to $\HGM((1/N,-1/N),(1,1)|t_0)$, for
$a=2-4t_0$.

In order to get a representation over the field
$\Q[\zeta_N+\zeta_N^{-1}]$, we need to get an extra splitting of the
Jacobian.  The curve $\D_N$ has once again two different involutions,
namely the canonical involution $\tau$ and the involution
$\iota_N:(x,y) \to \left(\frac{1}{x},\frac{y}{x^{N+1}}\right)$. Both
involutions commute with each other; $\tau$ commutes with the action
of the $N$-th roots of unity, but $\iota$ does not, they satisfy the relation
$$\iota \circ \zeta_N = \zeta_N^{-1} \circ \iota.$$

Let $g(x)$ denote the monic
polynomial whose zeroes are all the numbers $\xi + \xi^{-1}$, for
$\xi \in \overline{\Q}\setminus\{-1\}$ satisfying
$\xi^{N/2}=-1$. Clearly
\[
  \deg(g(x)) = \begin{cases}
    (N/2-1)/2 & \text{ if }4 \nmid N,\\
    N/4 & \text{ if } 4 \mid N.
    \end{cases}
  \]
\begin{lemma}
  The quotient of the curve $\D_N$ by the involution $\iota_N$ is
  given by the curve
\begin{equation}
  \label{eq:quot-2-odd}
\C_N: y^2=xg(x^2-2)+a,
\end{equation}
when $N/2$ is odd, and by the equation
\begin{equation}
  \label{eq:quot-2-even}
\C_N: y^2=(x+2)(xg(x^2-2)+a),
\end{equation}
when $N/2$ is even. The Jacobian of both curves contain
$\ZZ[\zeta_N+ \zeta_N^{-1}]$ in their endomorphism ring.
\end{lemma}
\begin{proof}
  See \cite[Proposition 3]{MR1138583} for the first statement. The
  second one is clear.
\end{proof}

Just for completeness, we compute the quotient of $\D_N$ by the involution
$\iota \circ \tau$.

\begin{lemma}
  The quotient of the curve $\D_N$ by the involution $\iota_N \circ \tau$ is
  given by the curve
\[
 y^2=x(x^2-4)(g(x^2-2)+a),
\]
when $N/2$ is odd, and by the equation
\[
y^2=(x-2)(xg(x^2-2)+a),
\]
when $N/2$ is even. The Jacobian of both curves contain
$\ZZ[\zeta_N+ \zeta_N^{-1}]$ in their endomorphism ring.
\end{lemma}
\begin{proof}
  Follows from a similar computation as the previous ones. In the first
  case, the functions on $\D_N$ fixed by the involution are generated
  by $\xi = x+\frac{1}{x}$ and $\eta=y(1-\frac{1}{x^2})x^{-(N/2-1)/2}$
  while in the second case, they are generated by
  $\xi = x+\frac{1}{x}$ and $\eta = y(1-\frac{1}{x})x^{-N/4}$.
\end{proof}

\begin{corollary}
  \label{coro:comparison-Darmon-2}
  Let $p$ be a prime number and let $t_0 \in \Q\setminus \{0,1\}$. The
  hypergeometric motive $\HGM((\frac{1}{2p},-\frac{1}{2p}),(1,1)|t_0)$ matches the
  quadratic twist by $(-1)$ of the hypergeometric curve denoted
  $C_p^-$ in \cite{Darmon}.
\end{corollary}

\begin{proof}
  Lemma~\ref{lemma:quad-char} implies that the motive
  $\HGM((\frac{1}{2p},-\frac{1}{2p}),(1,1)|t_0)$ is a quadratic twist by $(-1)$ of
  Euler's curve.  The curve $\C_{2p}$ matches the curve $C_p^-$ of
  \eqref{eq:quot-2-odd} with parameter $a=2-4t_0$.
\end{proof}

  \bibliographystyle{plain}
\bibliography{biblio}
\end{document}